\documentclass[12pt,reqno]{amsart}
\usepackage{graphicx} 
\usepackage{longtable} 
\usepackage{makecell}  
\usepackage{multirow}
\usepackage{enumerate}
\usepackage{amsmath, amsthm} 
\usepackage{amsfonts}
\usepackage{amssymb}
\usepackage{fullpage} 
\usepackage{xcolor}
\usepackage{url}

\newtheorem{definition}{Definition}[subsection]
\newtheorem{example}[definition]{Example}

\newtheorem*{theorem*}{Theorem}
\newtheorem{theoremIntro}{Theorem} 
\newtheorem{theorem}{Theorem}[section] 
\newtheorem{remark}[theorem]{Remark}
\newtheorem{lemma}[theorem]{Lemma}
\newtheorem{proposition}[theorem]{Proposition}
\newtheorem{corollary}[theorem]{Corollary}

\usepackage{multirow}
\usepackage{mathabx,amsmath,amscd, latexsym, amssymb, bbold}
 
\usepackage[numbers]{natbib}
\usepackage{amsmath, amssymb, tikz}

\newcommand{\C }{\mathcal{C}}
\newcommand{\B }{\mathcal{B}}
 \newcommand{\FPdim}{\text{FPdim}}
 \newcommand{\FQ }{\mathbb{Q}}
\newcommand{\g}{\mathfrak{g}}
\newcommand{\Y }{\mathcal{Z}}
\newcommand{\Z }{\mathbb{Z}}

\title{Realizing modular data from  centers of near-group categories}
\date{}

\author[yangzhou]{Zhiqiang Yu\textsuperscript{1}}
\email{zhiqyumath@yzu.edu.cn}

\author[ucsb]{Qing Zhang\textsuperscript{2}}
\email{qingzhang@ucsb.edu}

\address[1]{School of Mathematical Science, Yangzhou University, Yangzhou, 225002,  China}
\address[2]{Department of Mathematics, University of California, Santa Barbara, CA 93106, USA}

\begin{document}
\begin{abstract}
    In this paper, we show the existence of a near-group category of type $\mathbb{Z} / 4\mathbb{Z} \times \mathbb{Z} / 4\mathbb{Z}+16$ and compute the modular data of its Drinfeld center. We prove that a modular data of rank $10$ can be obtained through condensation of the Drinfeld center of the near-group category $\mathbb{Z} / 4\mathbb{Z} \times \mathbb{Z} / 4\mathbb{Z}+16$,  and it can also be realized as  the Drinfeld center of a fusion category of rank $4$. Moreover, we  compute the modular data for the Drinfeld center of a near-group category $\mathbb{Z} / 8\mathbb{Z}+8$ and show that the non-pointed factor of its condensation has the same modular data as the quantum group category $\mathcal{C}(\mathfrak{g}_2, 4)$.
\end{abstract}
\maketitle

\section{Introduction}
 Modular categories over $\mathbb C$ are non-degenerate ribbon fusion categories over $\mathbb C$.  These categories originally appeared in conformal field theory \cite{moore1989classical} and were formally defined in \cite{turaev1992modular} to study (2+1)-dimensional topological quantum field theories. Algebraically, modular categories arise as representation categories of  quantum groups, von Neumann algebras, vertex operator algebras, and local conformal nets. 
 
 One of the central problems in the study of modular categories is their classification. As shown in \cite{bruillard2016rank}, there are finitely many classes of modular tensor categories of a given rank (i.e., the number of isomorphism classes of simple objects) up to tensor equivalence. This finiteness property motivates the rank classification of modular categories.  The classification of low-rank modular categories has been achieved in \cite{RSW09,BNRWbyrank2016, NRWWrank6, ng2023classification}, at least up to their modular data.

 A key invariant used in the classification of modular categories is their modular data, which is a pair of $S$ and $T$ matrices. Given any modular category $\mathcal C$, the modular data defines a projective representation of the modular group $\operatorname{SL}(2, \mathbb Z)$. The kernel of this projective representation is a congruence subgroup of level $N$ thus factors through $\operatorname{SL}(2, \mathbb Z/N\mathbb Z)$, where  $N$  is the order of the  $T$  matrix \cite{ng2010congruence}. Furthermore, this projective representation can be lifted to  linear representations, which also have congruence kernels \cite{dong2015congruence}. 
 
In \cite{NRWWrank6}, a method for classifying modular categories was introduced by constructing modular data directly from the irreducible representations of $\operatorname{SL}(2, \mathbb{Z}/n\mathbb{Z})$ for certain integers $n$. Importantly, irreducible representation of  $\operatorname{SL}(2, \mathbb{Z}/n\mathbb{Z})$ are completely classified in \cite{Nobs1,Nobs2}.  This technique was successfully applied in \cite{NRWWrank6} to classify rank $6$ modular categories up to their modular data. The classification was further extended in \cite{ng2023classification}, where  all modular data up to rank $11$ are explicitly constructed. 

One of the main motivations of this paper is to find a modular category that realize a rank $10$ modular data constructed in  \cite{ng2023classification}.  It was conjectured in \cite{ng2023classification} that this modular data could be realized through the condensation of the Drinfeld center of a near-group category of type  $\mathbb Z/4\mathbb Z\times \mathbb Z/4\mathbb  Z + 16$, provided such a category exists. 

To achieve this, we first prove the existence of such near-group category using results from \cite{izumi2001structure, EvansGannon}.  The next step is deriving the modular data for the center of $\mathbb Z/4\mathbb Z\times \mathbb Z/4\mathbb  Z+ 16$,  a category of rank $304$. One of the computational main challenges is to solve a large system of non-linear equations in \cite{izumi2001structure}, which determine the modular data of the center of the near-group category of $G+n$, where $n$ is the order of $G$. Through this process, we obtain the explicit modular data of the center. In the center of this near-group category, there exists a Tannakian fusion subcategory $\mathcal{E} =\operatorname{Rep}(\mathbb{Z}/2 \mathbb{Z} \times \mathbb{Z}/4 \mathbb{Z})$. We show that the condensation processed by $\mathcal{E}$ yields a   modular category of rank $10$. By further analyzing the congruence representations of $\operatorname{SL}(2, \mathbb Z)$, we identify the resulting rank $10$ modular data with that from  \cite{ng2023classification}. Explicitly, we have the following result:

\begin{theoremIntro}
Let $\mathcal C$ be a near-group category of type $\mathbb Z/4\mathbb Z\times \mathbb Z/4\mathbb  Z +16$ associated with the symmetric bi-character $\langle(g_1, g_2),(h_1, h_2)\rangle=\left(\zeta_4\right)^{g_1 h_1- g_2h_2}$, $g_1,g_2,h_1,h_2\in \mathbb Z/4\mathbb Z\times \mathbb Z/4\mathbb Z$.Then $\mathcal Z(\mathcal{C})$ contains a Tannakian fusion subcategory $\mathcal{E}=\operatorname{Rep}(G)$, where $G=\mathbb{Z} / 2 \mathbb{Z} \times \mathbb{Z} / 4 \mathbb{Z}$. Moreover, the condensation $\mathcal Z(\mathcal{C})_{G}^0$  has the same modular data as in Equation (\ref{eq:MDrank10}).
\end{theoremIntro}

Several low-rank modular categories can be constructed using similar methods, such as by factorization and/or condensation of the centers of near-group categories. See \cite{ng2023classification, rowell2025neargroup} for examples of verified cases and conjectures.  However, it remains unclear whether some of these categories can be realized through quantum group categories,  vertex operator algebras or their condensations. 
  
  To verify a conjecture from \cite[Section 4]{rowell2025neargroup}, we compute the center of a near-group category of type  $\mathbb Z/8\mathbb Z + 8$. The existence of this fusion category was established in \cite{EvansGannon}. By computing and analyzing the modular data of the center, we obtain the following result:
    
\begin{theoremIntro}
   
Let $\mathcal{C}$ be a near-group  category of type $\mathbb Z/8\mathbb Z+8$.   Then $\mathcal{Z}(\mathcal{C})$ contains a Tannakian subcategory $\mathcal{E}=\operatorname{Rep}(\mathbb Z/2\mathbb Z)$ such that there is a braided tensor equivalence $\mathcal{Z}(\mathcal{C})_{\mathbb Z/2\mathbb Z}^0\cong\mathcal{C}(\mathbb Z/4\mathbb Z,q)\boxtimes \mathcal{D}$, where $\mathcal{C}(\mathbb Z/4\mathbb Z,q)$ is a pointed modular category and  $\mathcal D$ has the same modular data as $\mathcal{C}(\mathfrak{g}_2,4)$.      
    \end{theoremIntro}

The paper is organized as follows. In Section \ref{Preliminary}, we review some basic notions of modular categories. In Section \ref{section3}, we first show the existence of a near-group category $\mathcal{C}$ of type $\mathbb Z/4\mathbb Z\times \mathbb Z/4\mathbb Z+16$ (Proposition \ref{prop:existenceZ4*Z4}).  We then compute the modular data of  $\mathcal{Z}(\mathcal{C})$, and prove that the condensation by Tannakian fusion category $\operatorname{Rep}(\mathbb Z/2\mathbb Z\times\mathbb Z/4\mathbb Z)$ matches the modular data listed in \cite{ng2023classification} (Theorem \ref{case(3,3,1,1,1,1)}), and we show the condesation is also the Drinfeld center of a self-dual fusion category of rank $4$ (Proposition \ref{rank4category}). Finally, we compute the Drinfeld center of a near-group category of type $\mathbb Z/8\mathbb Z+8$  and  obtain the modular data for the non-pointed factor from the condensation of the  center (Theorem \ref{ModdatacondeZ8}). 

Throughout this paper, we use the notation $\mathbb{T}=\{x \in \mathbb{C} ;|x|=1\}$, $\zeta_n = \exp \left(2\pi i/n\right)$ and $\chi_n^m=m+\sqrt{n}$.

\section{Preliminary}\label{Preliminary}

In this section, we recall some basic  notions on braided fusion categories.  For further details, readers are referred to \cite{bakalov2001lectures, EGNO, izumi2001structure, EvansGannon}. 
\subsection{Modular data}
Let $\mathcal{C}$ be a braided fusion category with a braiding $c$. We use $\mathcal O(\mathcal C)$ to denote the set of isomorphism classes of simple objects in $\mathcal C$. Let $\mathcal{D}$ be a braided fusion subcategory of $\mathcal{C}$, the  \textit{M\"{u}ger centralizer} $\mathcal{D}_\mathcal{C}'$ of $\mathcal{D}$ in $\mathcal{C}$ is the fusion subcategory generated by objects $X$ of $\mathcal{C}$ such that $c_{Y,X}c_{X,Y}=\operatorname{id}_{X\otimes Y}$ for all $Y$ of $\mathcal{D}$ \cite{Muger2003}. The \textit{M\"{u}ger center} of $\mathcal{C}$ is the fusion subcategory $\mathcal C' =\mathcal{C}_\mathcal{C}'$. A braided fusion category $\mathcal{C}$ is \textit{symmetric} if $\mathcal{C}'=\mathcal{C}$. We say a symmetric fusion category $\mathcal{C}$ is Tannakian if $\mathcal{C}\cong \operatorname{Rep}(G)$ as braided fusion category, where $G$ is a finite group.

A ribbon fusion category $\mathcal C$  is called \textit{modular}, if its M\"{u}ger center is equivalent to $\operatorname{Vec}$, the category of finite-dimensional vector space over $\mathbb{C}$. Equivalently,  $\mathcal C$ is modular if the $S$-matrix is non-degenerate \cite{Muger2003}, where $S=\operatorname{Tr}(c_{Y,X^\ast}c_{X^\ast,Y})$, $X,Y\in\mathcal{O}(\mathcal{C})$. Let $T=(\delta_{X,Y}\theta_X)$, where $\theta_X$ is the scalar associated with the ribbon twist of $X\in\mathcal O(\mathcal C)$. 
The pair of $(S,T)$ is referred as the \textit{modular data}. Let $\mathfrak{s}, \mathfrak{t}\in\operatorname{SL}(2, \mathbb{Z})$ be the generators given by 
\begin{align*}\mathfrak{s}:=   \begin{pmatrix} 0 & -1\\1 &0\end{pmatrix},~ \mathfrak{t}:=   \begin{pmatrix} 1 & 1\\0 &1\end{pmatrix}
\end{align*}
with relations $\mathfrak{s}^4 =1$ and $(\mathfrak{s}\mathfrak{t})^3=\mathfrak{t}^2$.
For a modular tensor category $\mathcal{C}$, the assignment  \begin{align*}
    \mathfrak{s}\mapsto \frac{1}{\sqrt{\dim(\mathcal{C})}}S,\quad\mathfrak{t}\mapsto T
\end{align*}  gives a projective representation of $\operatorname{SL}(2, \mathbb{Z})$. 
Let $N$ be the order of the $T$ matrix. The kernel of this projective representation is a congruence subgroup of 
 $\operatorname{SL}(2,\mathbb Z)$ at level $N$. Consequently,  the image factors through the finite group $\operatorname{SL}(2,\mathbb Z/N\mathbb Z)$ \cite{ng2010congruence}. This projective representations can be lifted to linear representations, which  also have congruence kernels \cite{dong2015congruence}. The congruence kernel results have been used in proving key results such as rank finiteness of modular categories \cite{bruillard2016rank}, and in the classifications of modular categories \cite{BNRWbyrank2016, NRWWrank6, ng2023classification, Ngtransitive2022, plavnik2023modular}.

Let $\rho$ be an arbitrary irreducible finite-dimensional congruence representation of $\operatorname{SL}(2,\mathbb Z)$ of level $n$, where $n$ is a positive integer. By the Chinese Remainder Theorem,  $\rho$  factors through the finite groups
 \begin{align*}
\operatorname{SL}(2,\mathbb Z/n\mathbb Z)\cong\operatorname{SL}\big(2,\mathbb Z/p_1^{n_1}\mathbb Z\big)\times\cdots\times\operatorname{SL}\big(2,\mathbb Z/p_r^{n_r}\mathbb Z\big),
\end{align*} 
where $n=p_1^{n_1}\cdots p_r^{n_r}$, and $p_j$ are distinct primes, $1\leq j\leq r$. The finite-dimensional  irreducible representations of $\operatorname{SL}(2,\mathbb{Z}/p^m\mathbb{Z})$ are classified in \cite{Nobs1,Nobs2}.

Let $\mathcal C$ be a ribbon fusion category with $S$- and $T$- matrix. For any simple objects $X$ and $Y$, we have the balancing equation \cite{bakalov2001lectures}
$$\theta_X \theta_Y S_{X, Y}=\sum_{Z \in \mathcal O(\mathcal C)} N_{X^ \ast, Y}^Z \theta_Z d_Z,$$
where $N_{X^ \ast, Y}^Z   =\operatorname{dim} \operatorname{Hom}\left(X^\ast \otimes Y, Z\right)$ is the fusion coefficient and $d_Z$ is the quantum dimension of the simple object $Z$. If $\mathcal C$ is modular, the fusion coefficients can be expressed in terms of the $S$-matrix via the Verlinde formula (see, for example, \cite{bakalov2001lectures})
$$N_{X, Y}^Z=\frac{1}{\operatorname{dim}(\mathcal{C})} \sum_{W \in \mathcal O (\mathcal{C})} \frac{S_{X, W} S_{Y, W} S_{Z^\ast, W}}{S_{\mathbf 1, W}}  ,$$
where $\mathbf 1$ denotes the tensor unit of $\mathcal C$.

Let $\mathcal C$ be modular with modular data $(S,T)$. 
Denote by $\mathbb{Q}(S)$ the smallest field containing all entries of the $S$-matrix. 
It is known that $\mathbb{Q}(S)$ is an abelian Galois extension over $\mathbb{Q}$ \cite{dBG, CosteGannon}. 
Let $\mathbb{Q}_N:=\mathbb{Q}(\zeta_N)$, where $ \zeta_N $ is a primitive $ N$th root of unity. 
It is proved in \cite{ng2010congruence} that $\mathbb{Q}(S) \subset \mathbb{Q}(T)=\mathbb{Q}_N$, where $N$ is the order of the $T$ matrix.  
Let $\operatorname{Gal}(\mathcal C)=\operatorname{Gal}(\mathbb{Q}(S)/\mathbb{Q})$ denote the Galois group of the field extension $\mathbb{Q}(S)/\mathbb{Q}$.  
 As $$\frac{S_{X,Y} S_{X,Z}}{S_{\mathbf 1,X}^2}= \sum_{W\in \mathcal{O}(\mathcal C)}N_{YZ}^W\frac{S_{XW}}{S_{\mathbf 1,X}},$$ \cite[Lemma 2.4]{Muger2003} and $S$ is symmetric, the assignments $ \chi_Y:X\mapsto \frac{S_{X,Y}}{S_{\mathbf 1,Y}} $ define linear characters of the fusion ring  $ \mathcal{K}_0(\mathcal C) $.
Since $ S$  is non-degenerate, $ \{\chi_Y\}_{Y\in\mathcal O(\mathcal C)} $  forms all characters of $ \mathcal{K}_0(\mathcal C) $. 
For any $ \sigma\in \operatorname{Gal}(\mathcal C) $,  $ \sigma(\chi_Y) $  is also a linear character of $ \mathcal{K}_0(\mathcal C) $.  
Thus, there is a unique $ \hat{\sigma} \in \operatorname{Sym}(\mathcal{O}(\mathcal{C}))$ such that 

$$\sigma\left(\dfrac{S_{X,Y}}{S_{\mathbf 1,Y}}\right)=\dfrac{S_{X,\hat{\sigma}(Y)}}{S_{\mathbf 1,\hat{\sigma}(Y)}},$$
for all $ X, Y \in \mathcal{O}(\mathcal{C})$.  For a more in-depth discussion on the Galois action on modular data, see \cite{bruillard2016rank, dong2015congruence} and the sources cited therein.

 A fusion category  $\mathcal{C}$  is called \textit{pointed} if all of its simple objects are invertible. In this case, the set of isomorphism classes of simple objects  $\mathcal{O}(\mathcal{C})$  forms a finite group $G$.
 Given a braided pointed fusion category $\mathcal{C}$ with $\mathcal{O}(\mathcal{C})=G$,  the braiding  $c$  naturally determines a quadratic form  $q$  on  $G$, given by  $q(g) = c_{g,g}$. The pair $(G,q)$ is  called a \textit{pre-metric group}. In fact, there is a bijective correspondence between pre-metric groups and ribbon pointed fusion categories \cite{Drinfeld2010}. We denote by $\mathcal{C}(G,q)$ the ribbon pointed fusion category associated with the pre-metric group $(G,q)$. A pre-metric group $(G,q)$ is called a \textit{metric group} if the quadratic form $q$ is non-degenerate, that is, the bi-character   
 \begin{align*}\langle g,h\rangle:=\frac{q(g+h)}{q(g)q(h)},~ \forall g,h\in G
 \end{align*}
 determined by $q$ is non-degenerate. In this case, the corresponding  category is modular.

\subsection{Near-group categories}\label{subsectionNeargroup} 
 Let $G$ be a finite abelian group of order $n$, and let $m$ be a non-negative integer. A \textit{near-group category} of type $G+m$ \cite{Siehler2003} is a fusion category with simple objects labeled by elements $g\in G$ with an additional simple object  $\rho$  such that the fusion rules are generated by the multiplication in $G$, with \begin{align*}
  g\rho=\rho g=\rho,\, \forall g\in G,~ \rho^{\otimes2} = m\rho\oplus_{g\in G}g  . 
 \end{align*}As shown in \cite[Theorem 2]{EvansGannon} and \cite[Appendix A]{ostrik2015pivotal}, the only possible values of $m$ are $n$, $n-1$ or $m = kn$ for some nonnegative integer $k$. When $m=0$, the corresponding near-group categories  are  called the Tambara-Yamagami  categories, which are  classified in \cite{TY98}. For further developments in the classification of near-group categories, we refer the reader to \cite{Budthesis,ostrik2015pivotal,Schopieray2023} and the references therein.
  
 When $m=n$, the near-group categories $G+n$ are characterized by the following:
    
\begin{theorem} \cite[Theorem 5.3]{izumi2001structure}, \cite[Corollary 5]{EvansGannon}\label{thm:existng}
Let $G$ be a finite abelian group of order $n$, $\langle \cdot , \cdot\rangle$ a non-degenerate symmetric bicharacter on $G$ and define $d:=\dfrac{n+\sqrt{n^2+4n}}{2}$. Let  $c\in \mathbb T$, $a: G\to \mathbb T$, $b: G\to \mathbb C$ be such that 
\begin{equation}\label{eq:ng1}
     a(0)=1, \quad a(x)=a(-x), \quad a(x+y)\langle x, y\rangle=a(x)a(y),\quad \sum_{a\in G} a(x)=\sqrt{n}c^{-3},
\end{equation}

\begin{equation}\label{eq:ng2}
    b(0)=-\frac{1}{d}, \quad \sum_y \overline{\langle x, y\rangle} b(y)=\sqrt{n} c \overline{b(x)}, \quad a(x) b(-x)=\overline{b(x)},
\end{equation}
\begin{equation}\label{eq:ng3}
    \sum_x b(x + y) \overline{b(x)}=\delta_{y, 0}-\frac{1}{d}, \quad \sum_x b(x + y) b(x + z) \overline{b(x)}=\overline{\langle y, z\rangle} b(y) b(z)-\frac{c}{d \sqrt{n}}.
\end{equation}

Then the data $\langle \cdot ,\cdot \rangle$, $  c, a, b$ determine a near-group category of type $G+n$. Two such categories $\mathcal{C}_1$ and $\mathcal{C}_2$ determined by $\langle \cdot ,\cdot \rangle_1,c_1, a_1, b_1$ and $\langle \cdot ,\cdot \rangle_2,c_2, a_2, b_2$  are equivalent as fusion categories if and only if there is $\psi \in \operatorname{Aut}(G)$ such that $\langle x, y\rangle_1=\langle\psi x, \psi y\rangle_2, a_1(x)=a_2(\psi x)$, $b_1(x)=b_2(\psi x)$ and $c_1=c_2$.
\end{theorem}

 The classification of near-group categories of type $G+n$ is complete for $n\leq 13$ \cite{EvansGannon, izumi2001structure}, where $n$ is the order of $G$. This classification has been further extended in \cite{Budthesis} for cyclic groups of order less than $29$.

Once the existence of  a near-group category of type $G+n$ is established, the modular data for the Drinfeld center of such a category is given as follows \cite{izumi2001structure}.  Specifically, we need to find all functions $\xi: G \rightarrow \mathbb{T}$ and values $\tau \in G, \omega \in \mathbb{T}$ which satisfy 
\begin{equation}\label{eq:half1}
    \sum_g \xi(g)=\sqrt{n} \omega^2 a(\tau) c^3-n d^{-1},
\end{equation}
\begin{equation}\label{eq:half2}
    \bar{c} \sum_k b(g+k) \xi(k)=\omega^2 c^3 a(\tau) \overline{\xi(g+\tau)}-\sqrt{n} d^{-1},
\end{equation}
\begin{equation}\label{eq:half3}
    \xi(\tau-g)=\omega c^4 a(g) a(\tau-g) \overline{\xi(g)},
\end{equation}
\begin{equation}\label{eq:half4}
    \sum_k \xi(k) b(k-g) b(k-h)=c^{-2} b(g+h-\tau) \xi(g) \xi(h) \overline{a(g-h)}-c^2 d^{-1}.
\end{equation}

There are $n(n+3)/2$ triples $(\xi_i, \tau_i, \omega_i)$ that satisfy the above equations.
The corresponding center has rank $n(n+3)$ with the following 4 subsets of simple objects:
\begin{enumerate}
    \item $\mathfrak{a}_g, g \in G$, $\dim(\mathfrak{a}_g)=1$;
    \item $\mathfrak{b}_h, h \in G$, $\dim(\mathfrak{b}_h)=d+1$;
    \item $\mathfrak{c}_{l, k}=\mathfrak{c}_{k, l}, l, k \in G, l \neq k$, $\dim(\mathfrak{c}_{k,l})=d+2$
    \item $\mathfrak{d}_j$, where $j$ corresponds to a triple $\left(\xi_j, \tau_j, \omega_j\right)$, $\dim(\mathfrak{d}_j)=d$.
\end{enumerate}
The $T$  and $S$ matrices are given by the following block form
\begin{equation}\label{Tmatrix-center}
    T=\operatorname{diag}\left(\langle g, g\rangle,\langle h, h\rangle,\langle k, l\rangle, \omega_j\right)
\end{equation}

\begin{equation}\label{Smatrix-center}
    S=\\ \left(\begin{array}{cccc}\left\langle g, g^{\prime}\right\rangle^{-2} & ( d+1)\left\langle g, h^{\prime}\right\rangle^{-2} & ( d+2) \overline{\left\langle g, k^{\prime}+l^{\prime}\right\rangle} &  d\left\langle g, \tau_{j^{\prime}}\right\rangle \\( d+1)\left\langle h, g^{\prime}\right\rangle^{-2} & \left\langle h, h^{\prime}\right\rangle^{-2} & ( d+2) \overline{\left\langle h, k^{\prime}+l^{\prime}\right\rangle} & - d\left\langle h, \tau_{j^{\prime}}\right\rangle \\( d+2) \overline{\left\langle k+l, g^{\prime}\right\rangle} & ( d+2) \overline{\left\langle k+l, h^{\prime}\right\rangle} & S_{(k, l),\left(k^{\prime}, l^{\prime}\right)} & \mathbf 0 \\ d\left\langle\tau_j, g^{\prime}\right\rangle & - d\left\langle\tau_j, h^{\prime}\right\rangle &\mathbf 0 & S_{j, j^{\prime}}\end{array}\right),
\end{equation}
where 
\begin{equation}\label{Smatrix-S44}
    \begin{aligned} & S_{j, j^{\prime}}=\omega_j \omega_{j^{\prime}} \sum_{g \in G}\left\langle\tau_j+\tau_{j^{\prime}}+g, g\right\rangle \\ & + d \omega_j \omega_{j^{\prime}}c^6 a\left(\tau_j\right) a\left(\tau_{j^{\prime}}\right) n^{-1} \sum_{g, h \in G} \overline{\xi_j(g) \xi_{j^{\prime}}(h)\left\langle\tau_{j}-\tau_{j^\prime}+h-g, h-g\right\rangle}, \end{aligned}
\end{equation}
and 
\begin{equation}\label{Smatrix-S33}
    S_{(k, l),\left(k^{\prime}, l^{\prime}\right)}=(d+2)\left(\overline{\left\langle k, k^{\prime}\right\rangle\left\langle l, l^{\prime}\right\rangle}+\overline{\left\langle k, l^{\prime}\right\rangle\left\langle l, k^{\prime}\right\rangle}\right).
\end{equation}

\subsection{Condensation of modular categories}
Let $\mathcal{C}$ be a modular  tensor category. Assume that $\mathcal{C}$ contains  a Tannakian fusion subcategory $\text{Rep}(G)$, then the de-equivariantization  of $\mathcal{C}$ by $\operatorname{Rep}(G)$, denoted by $\mathcal{C}_G$, is a $G$-crossed braided fusion category with a faithful $G$-grading, see \cite[section 4.4]{Drinfeld2010}. Moreover, the trivial component $\C_G^0$  with respect to the associated $G$-grading is also a modular category \cite[Proposition 4.56]{Drinfeld2010}, which is called the condensation of $\C$ by $\text{Rep}(G)$ (see \cite{Drinfeld2010, Bruguieres2000, KirillovOstrik2001} for further details). $\mathcal{C}_G^0$   is also  braided tensor equivalent to the de-equivariantization of centralizer of $\operatorname{Rep}(G)$ in $\mathcal{C}$ by $\operatorname{Rep}(G)$. The dimension of  $\mathcal{C}_G^0$ is given by $\dim(\mathcal{C}_G^0)=\frac{\dim(\mathcal{C})}{|G|^2}$ \cite{Drinfeld2010}. 
In addition, let $A=\text{Fun}(G)\in \operatorname{Rep}(G)$ be the regular algebra of $\operatorname{Rep}(G)$, then it is a connected \'{e}tale algebra in sense of \cite{DMNO,KirillovOstrik2001}, then  the category of local  modules $\mathcal{C}_A^0$ over  the \'{e}tale algebra  in $\mathcal{C}$ is exactly the modular category $\mathcal{C}_G^0$. 
    
    Given a fusion category $\mathcal{C}$, let $\chi$ be    an irreducible character of $\mathcal{K}_0(\mathcal{C})$. Then there is central element $c_\chi$ in $\mathcal{K}_0(\mathcal{C})\otimes_\mathbb{Z}\mathbb{C}$, which acts on $\chi$ as a scalar $f_\chi$, and  $f_\chi$  is called a formal codegree determined by $\chi$, see \cite{ostrik2015pivotal}.
\begin{proposition}\label{equivariantization}
    Let $\mathcal{C}$ be a near-group category of type $G+n$, where $G$ is abelian. If $\mathcal{Z}(\mathcal{C})$ contains a non-trivial Tannakian subcategory $\operatorname{Rep}(H)$ with $H$ being abelian, then $\mathcal{C}$ is the equivariantization of a  fusion category $\mathcal{D}$. Moreover, $\mathcal{Z}(\mathcal{C})_H^0\cong\mathcal{Z}(\mathcal{D})$.
\end{proposition}
\begin{proof}
Let $F:\mathcal{Z}(\mathcal{C})\to\mathcal{C}$ be the forgetful functor, and $I$ be the adjoint functor to $F$.    Then the dimensions of simple direct summands of $I(\mathbf{1})$ are  $\frac{\dim(\mathcal{C})}{f_\chi}$ by \cite[Theorem 2.13]{ostrik2015pivotal}, where $f_\chi$ are formal codegrees of $\mathcal{C}$  determined by an irreducible character of the Grothendieck ring $\mathcal{K}_0(\mathcal{C})$. It follows from \cite[Proposition A.5]{ostrik2015pivotal} that the only invertible simple object of $I(\mathbf{1})$ is the unit object, as $F$ is a tensor functor,  hence the Tannakian category $\operatorname{Rep}(H)$ is embedded into  $\mathcal{C}$ under the forgetful functor $F$. Then $\mathcal{C}$ is the equivariantization of fusion category $\mathcal{D}$ and $\mathcal{Z}(\mathcal{C})_H^0\cong\mathcal{Z}(\mathcal{D})$  as braided fusion category 
 by \cite[Proposition 2.10]{ENOweaklygrouptheoretical}. 
\end{proof}
The following example may be known to some experts, see \cite[Table 8]{rowell2025neargroup} for example, we include it here.
\begin{example}Let $\mathcal{C}$ be a near-group fusion category of type $\mathbb{Z}/4\mathbb{Z}+4$, then there exists a modular tensor equivalence  $\mathcal{Z}(\mathcal {C})_{\mathbb{Z}/2\mathbb{Z}}^0\cong\mathcal{C}(\mathbb{Z}/2\mathbb{Z},q)\boxtimes\mathcal{C}(\mathfrak{sl}_3,5)_{\operatorname{ad}}$, where
$\mathcal{C}(\mathfrak{sl}_3,5)_{\operatorname{ad}}$ is the adjoint subcategory of $\mathcal{C}(\mathfrak{sl}_3,5)$.
 \end{example}
Notice that the simple objects of $\mathcal{C}(\mathfrak{sl}_3,5)_{\operatorname{ad}}$ are labeled  by \begin{align*}
 \{(0,0),(0,3),(3,0),(1,1),(2,2),(1,4),(4,1)\}.
 \end{align*}
Due to conformal embedding $A_{2,5}\subseteq A_{5,1}$ \cite{DMNO},
it is well-known that $\mathcal{C}(\mathfrak{sl}_3,5)_{\operatorname{ad}}$ contains a connected \'{e}tale algebra $A=(0,0)\oplus(2,2)$ such that $\mathcal{C}(\Z/2\mathbb{Z},q)=(\mathcal{C}(\mathfrak{sl}_3,5)_{\operatorname{ad}})_A^0$, 
hence,  $\mathcal{C}(\Z/2\mathbb{Z},q)^{\operatorname{rev}}\boxtimes\mathcal{C}(\mathfrak{sl}_3,5)_{\operatorname{ad}}\cong\mathcal{Z}(\mathcal{B})$ for  fusion category $\mathcal{B}:=(\mathcal{C}(\mathfrak{sl}_3,5)_{\operatorname{ad}})_A$. Then $\mathcal{B}$ is a fusion category of rank $4$ and fusion rules are
 \begin{align*}
 X\otimes X^*=\mathbf{1}\oplus X\oplus X^*, g\otimes X=X^*=X\otimes g, \\
 g\otimes g=\mathbf{1},X\otimes X=X^*\otimes X^*=g\oplus X\oplus X^*.
 \end{align*}
 Indeed, let $X=F(\mathbf{1}\boxtimes(1,4))$, where $F:\mathcal{Z}(\mathcal{B})\to\B$ is the forgetful functor, then $X$ is a simple object of Frobenius-Perron dimension $1+\sqrt{2}$, hence
 \begin{align*}
 X\otimes X&=F(\mathbf{1}\boxtimes(1,4)\otimes \mathbf{1}\boxtimes(1,4))\\
 &=F(\mathbf{1}\boxtimes(3,0)\oplus \mathbf{1}\boxtimes(4,1))\\
 &=F(\mathbf{1}\boxtimes(3,0))\oplus F(\mathbf{1}\boxtimes(4,1)).
 \end{align*}
 Since $\FPdim((3,0))=2+\sqrt{2}$,  it is not simple by counting the Frobenius-Perron dimension of $\B$. Hence, $F(\mathbf{1}\boxtimes(3,0))$ must contains an invertible object $g$ as a direct summand. However, $\theta_{(3,0)}=\zeta_4^3$, $\mathbf{1}\boxtimes(3,0)\nsubseteq I(\mathbf{1})$, where $I$ is the adjoint functor to $F$, so $g\neq \mathbf{1}$. Thus $g\subseteq F(\mathbf{1}\boxtimes(3,0))$ and $X^*=F(\mathbf{1}\boxtimes(4,1))\subseteq X\otimes X^*$, which implies
 \begin{align*}
 X\otimes X=g\oplus X\oplus X^*=X^*\otimes X^*,
 \end{align*}
 then the rest fusion rules are easy. And it was proved that the $\mathbb{Z}/2\mathbb{Z}$-equivariantization of $\mathcal{B}$ is a near-group fusion category of type $\mathbb{Z}/4\mathbb{Z}+4$ \cite{Izumicuntzalgebra}.

\section{Realization of modular data}\label{section3}
In \cite{ng2023classification}, the following modular data of rank $10$ labeled as $10_{0,1435}^{20,676}$ was constructed: 
\small{\begin{equation}\label{eq:MDrank10}
  \begin{aligned}
      S&=\begin{pmatrix}
1 & \chi_{20}^5 & \chi_{20}^5 & \chi_{20}^5 & \chi_{20}^5 & \chi_{20}^5 & \chi_{20}^5 & 4  \chi_5^2  & 4  \chi_5^2  & \chi_{80}^9 \\
\chi_{20}^5 & 3\chi_{20}^5 & -\chi_{20}^5 & -\chi_{20}^5 & -\chi_{20}^5 & -\chi_{20}^5 & -\chi_{20}^5 & 0 & 0 & \chi_{20}^5 \\
\chi_{20}^5 & -\chi_{20}^5 & 3\chi_{20}^5 & -\chi_{20}^5 & -\chi_{20}^5 & -\chi_{20}^5 & -\chi_{20}^5 & 0 & 0 & \chi_{20}^5 \\
\chi_{20}^5 & -\chi_{20}^5 & -\chi_{20}^5 & s  \chi_{20}^5  & \bar{s}  \chi_{20}^5  & \chi_{20}^5 & \chi_{20}^5 & 0 & 0 & \chi_{20}^5 \\
\chi_{20}^5 & -\chi_{20}^5 & -\chi_{20}^5 & \bar{s}  \chi_{20}^5  & s  \chi_{20}^5  & \chi_{20}^5 & \chi_{20}^5 & 0 & 0 & \chi_{20}^5 \\
\chi_{20}^5 & -\chi_{20}^5 & -\chi_{20}^5 & \chi_{20}^5 & \chi_{20}^5 & \bar{s}  \chi_{20}^5  & s  \chi_{20}^5  & 0 & 0 & \chi_{20}^5 \\
\chi_{20}^5 & -\chi_{20}^5 & -\chi_{20}^5 & \chi_{20}^5 & \chi_{20}^5 & s  \chi_{20}^5  & \bar{s}  \chi_{20}^5  & 0 & 0 & \chi_{20}^5 \\
4  \chi_5^2  & 0 & 0 & 0 & 0 & 0 & 0 & -2  \chi_{5}^3  &  \chi_{180}^{14} & -4  \chi_5^2  \\
4  \chi_5^2  & 0 & 0 & 0 & 0 & 0 & 0 &  \chi_{180}^{14} & -2  \chi_{5}^3  & -4  \chi_5^2  \\
\chi_{80}^9 & \chi_{20}^5 & \chi_{20}^5 & \chi_{20}^5 & \chi_{20}^5 & \chi_{20}^5 & \chi_{20}^5 & -4  \chi_5^2  & -4  \chi_5^2  & 1 \\
\end{pmatrix}, \\
T&=\operatorname{diag}\left( 1,1,1,i,i,-i,-i,\zeta_5^2,\zeta_5^3,1\right),
  \end{aligned}  
\end{equation}}\normalsize{where $\chi_n^m=m+\sqrt{n}$ and  $s=-1+2i$.} It was conjectured in \cite{ng2023classification} that they can be realized from the condensation reductions of the Drinfeld center of the near-group category of type $\mathbb Z/4\mathbb Z\times \mathbb Z/4\mathbb Z+16$, with the symmetric bi-character $\langle(g_1, g_2),(h_1, h_2)\rangle=\left(\zeta_4\right)^{g_1 h_1- g_2h_2}$, $(g_1,g_2), (h_1,h_2)\in \mathbb Z/4\mathbb Z\times \mathbb Z/4\mathbb Z$.

In this section, we  prove this conjecture in  Theorem \ref{case(3,3,1,1,1,1)}. Additionally, we obtain the modular data for the condensation of the Drinfeld center of a near-group category of type $\mathbb Z/8\mathbb Z+8$ in Theorem \ref{ModdatacondeZ8}. This modular data is a Galois conjugate to that of the  modular  category $\mathcal{C}(\mathfrak{g}_2,4)$ (see Corollary \ref{realcondZ8}), confirming a conjecture from   \cite[Table 8]{rowell2025neargroup}. 
\subsection{Near-group categories $\mathbb Z/4\mathbb Z\times \mathbb Z/4\mathbb Z+16$}\label{sec:Z4Z4}
\subsubsection{Existence of $\mathbb{Z}/4 \mathbb{Z}\times\mathbb{Z}/4 \mathbb{Z} + 16$}
Let \( \langle  \cdot , \cdot \rangle \) be the non-degenerate symmetric bi-character on \( \mathbb Z/4\mathbb Z\times \mathbb Z/4\mathbb Z \) defined by:
\begin{equation}\label{eq:bicharacter}
 \begin{gathered}
\langle\cdot, \cdot\rangle: \mathbb Z/4\mathbb Z\times \mathbb Z/4\mathbb Z \times \mathbb Z/4\mathbb Z\times \mathbb Z/4\mathbb Z \rightarrow \mathbb{T} \\
\left\langle\left(g_1, g_2\right),\left(h_1, h_2\right)\right\rangle=\left(\zeta_4\right)^{g_1 h_1-g_2 h_2}
\end{gathered}
\end{equation}
One solution associated with $\langle\cdot, \cdot\rangle$ by solving the equations in Theorem \ref{thm:existng} is 
\small{\begin{equation}\label{eq:Z4Z4sol}
\begin{aligned}
a(g_1,g_2) &= \zeta_8^{3(g_1^2 - g_2^2)}, \quad (g_1,g_2) \in \mathbb{Z}/4\mathbb{Z} \times \mathbb{Z}/4\mathbb{Z}, \quad c = 1, \quad
b(0,0) = \frac{1}{4}(2 - \sqrt{5}), \quad\\
b(1,2) &= \overline{b(2,1)}, \quad b(2,0) = \overline{b(0,2)}, \quad b(1,3) = \frac{1}{8} \left( 1 - \sqrt{5} + i \sqrt{2(-1 + \sqrt{5})} \right), \\
b(1,0) &= \overline{b(0,1)}, \quad b(1,1) = b(2,2) = \frac{1}{4}, \quad b(0,2) = -\frac{i}{4},  \quad\\b(0,1) &= \frac{1}{32} \left( (2 - 4i) + (1 - 3i) \sqrt{2} - 2\sqrt{5} + (1 + i)\sqrt{10} \right. \\
& \quad \left. - 2(-1)^{1/8} 2^{1/4}(2 + \sqrt{2}) \sqrt{(-1 + \sqrt{2})(-1 + \sqrt{5})} \right), \\
b(2,1) &= \frac{1}{32} \left( (2 - 4i) - (1 - 3i) \sqrt{2} - 2\sqrt{5} - 2(-1)^{1/4}\sqrt{5} \right. \\
& \quad \left. - 2i \sqrt{-1 + \sqrt{5}} + 2^{3/4} \sqrt{(-4 + 3\sqrt{2})(-1 + \sqrt{5})} \right).
\end{aligned}
\end{equation}
}


\normalsize
 \noindent The rest values of $b$ on $\mathbb Z/4\mathbb Z\times \mathbb Z/4\mathbb Z$ can be computed via the formula $b(-x)=\overline{a(x) b(x)}$ from Equation (\ref{eq:ng2}).  Up to equivalence of fusion categories, the above solution is the unique one corresponds to the bicharater in Equation (\ref{eq:bicharacter}):
\begin{proposition}
\label{prop:existenceZ4*Z4}
    Let  $\langle\cdot, \cdot\rangle$ be the non-degenerate symmetric bicharcter on $\mathbb Z/4\times \mathbb Z/4$ defined in Equation (\ref{eq:bicharacter}). Then, up to equivalence of fusion categories, there exists a unique near-group category of type  $\mathbb{Z}/4 \times \mathbb{Z}/4 + 16 $ corresponding to this bicharacter. 
\end{proposition}
\begin{proof}
    Solving Equation (\ref{eq:ng1}) in Theorem \ref{thm:existng} with respect to the bicharacter in Equation (\ref{eq:bicharacter}),  we found $4$ solutions for the quadratic forms: 
    $$
a_1\left(g_1, g_2\right)=\zeta_8^{3\left(g_1^2-g_2^2\right)}, \quad a_2\left(g_1, g_2\right)=\zeta_8^{3 g_1^2+g_2^2}, \quad a_3\left(g_1, g_2\right)=\zeta_8^{-g_1^2+g_2^2}, \quad a_4\left(g_1, g_2\right)=\zeta_8^{-g_1^2-3 g_2^2} .
$$
Only the quadratic form $a_1$ leads to further  solutions for all the equations  in Theorem \ref{thm:existng}. Furthermore, the equation $\sum_{a \in G} a(x)=\sqrt{n} c^{-3}$ in Equation (\ref{eq:ng1}) implies $c^{-3}=1$. This yields four possible solutions for  b, all corresponding to  $c = 1$.  These solutions are listed in Table \ref{table:solforb}. By Equations (\ref{eq:ng2}) and (\ref{eq:ng3}), we have \begin{align*}
b(0,0) = \dfrac{1}{4}(2-\sqrt{5}),\quad b(-x)=\overline{a(x) b(x)},\quad b(x)=\frac{1}{\sqrt{n}} \exp (i j(x)), 
\end{align*}
for some  $j(x) \in[-\pi, \pi]$ and $x\in \mathbb Z/4\mathbb Z\times \mathbb Z/4\mathbb Z$. Thus it is suffices to present values in the first row of Table \ref{table:solforb} to fully determine $b$. In particular, $J_1$ in Table \ref{table:solforb} corresponds to the solution in Equation $\left(\ref{eq:Z4Z4sol}\right)$. Next, consider the automorphisms on $\mathbb Z/4\mathbb Z\times \mathbb Z/4\mathbb Z$ defined by 
\begin{align*}
\psi_1(g_1, g_2)= \begin{cases}(g_1, g_2) & \text { if } g_1=0 \text { or } 2 \\ ( -g_1, g_2) & \text { if } g_1=1 \text { or } 3\end{cases},
\end{align*}
$\psi_2(g_1, g_2)=  (-g_1,-g_2)$, and 
\begin{align*}
\psi_3(g_1, g_2)= \begin{cases}(g_1, g_2) & \text { if } g_2=0 \text { or } 2 \\ ( g_1, -g_2) & \text { if } g_2=1 \text { or } 3\end{cases},\end{align*} where $(g_1,g_2)\in\mathbb Z/4\mathbb Z\times \mathbb Z/4\mathbb Z$. It is straightforward to verify $\psi_1, \psi_2$ and $\psi_3$ establish  equivalences from $J_1$ to $J_2$, $J_3$ and $J_4$ as described in Theorem \ref{thm:existng}, respectively. Thus the result follows. 

\small{
\begin{table}[ht]
\centering
\begin{tabular}{|l|l|l|l|l|l|l|l|l|l|}
\hline id &  $j(0,1)$ & $j(0,2)$ & $j(1,0)$ & $j(1,1)$ & $j(1,2)$ & $j(1,3)$ & $j(2,0)$ & $j(2,1)$ & $j(2,2)$ \\
\hline $J_1$ &  1.99103 & 1.57080 & $-1.99103$ & 0 & $-2.44331$ & $ -2.23704$ & $-1.57080$ & 2.44331 & 0 \\
\hline $J_2$ &  1.99103 & 1.57080 & $-0.36516$ & $ 2.23704$  & $-3.05447$ & 0 & $-1.57080$ & 2.44331  & 0 \\
\hline $J_3$ &  0.36516 & 1.57080 & $-0.36516$ & 0 & $-3.05447$ & 2.23704 & $-1.57080$ & 3.05447 & 0 \\
\hline $J_4$ &  0.36516 & 1.57080 &$-1.99103$ & $ -2.23704$ & $-2.44331$ & 0 & $-1.57080$ & 3.05447 & 0 \\
\hline
\end{tabular}
    \caption{Four solutions for $b$ in Proposition \ref{prop:existenceZ4*Z4}}
    \label{table:solforb}
\end{table}}

\end{proof}

\begin{remark}
    We do not know if there are any near-group categories of type  $\mathbb Z/4\mathbb Z\times \mathbb Z/4\mathbb Z+16$ other than the one in Proposition \ref{prop:existenceZ4*Z4} up to  Galois conjugation and  equivalence of fusion categories. For the remainder of this paper, we will use $\mathbb{Z}/4\mathbb Z \times \mathbb{Z}/4\mathbb Z + 16$ to refer specifically to the category  in Proposition \ref{prop:existenceZ4*Z4}.
\end{remark}

\subsubsection{Modular data of the center}  By solving the Equations (\ref{eq:half1})--(\ref{eq:half4}), we found $152$ triples  of solutions for $\left(\xi_i, \tau_i, \omega_i\right)$. Apply Equations $\left(\ref{Tmatrix-center}\right)$ -- $\left(\ref{Smatrix-S33}\right)$, we obtain the modular data of the center of $\mathbb{Z} / 4 \mathbb{Z} \times \mathbb{Z} / 4 \mathbb{Z}+16$. 
 In particular, the center of the near-group category $ \mathbb{Z} / 4 \mathbb{Z} \times \mathbb{Z} / 4 \mathbb{Z}+16 $ is a modular category of rank $304$ with the following simple objects:
 \begin{itemize}
    \item 16 invertible objects $X_g, g \in \mathbb{Z} / 4 \mathbb{Z} \times \mathbb{Z} / 4 \mathbb{Z}$;
\item  16 $\chi_{80}^9$-dimensional simple objects $Y_h, h \in  \mathbb{Z} / 4 \mathbb{Z} \times \mathbb{Z} / 4 \mathbb{Z} $;
\item 120 $\chi_{80}^{10}$-dimensional objects $Z_{k, l}, k, l \in \mathbb{Z} / 4 \mathbb{Z} \times \mathbb{Z} / 4 \mathbb{Z}$, and $k\neq l$;
\item  152 $\chi_{80}^8$-dimensional objects $W_{\omega_i, \tau_i, \xi_i}$, where $\left(\omega_i, \tau_i, \xi_i\right)$ are solutions in Table \ref{table:solZ4Z4}.
 \end{itemize}
 
 We provide detailed modular data in the Mathematica notebook \texttt{MDZ4Z4.nb} in the arXiv source. The pointed subcategory is essential for the next section, and the relevant data is presented below. Using Equations $\left(\ref{Tmatrix-center}\right)$ and $\left(\ref{Smatrix-center}\right)$, the  $S$  and  $T$ matrices for the pointed subcategory of  the center are given by:
\begin{equation}\label{eq:pointedST}
\begin{aligned}
& S_{\mathrm{pt}}= \left(\zeta_4^{g_1 h_1-g_2 h_2}\right), \text { where }\left(g_1, g_2\right),\left(h_1, h_2\right) \in \mathbb{Z} / 4 \mathbb{Z}\times \mathbb{Z} / 4 \mathbb{Z}, \\
& T_{\mathrm{pt}}=\operatorname{diag}\left(\zeta_4^{g_1^2-g_2^2}\right), \text { where }\left(g_1, g_2\right) \in \mathbb{Z} / 4 \mathbb{Z}\times \mathbb{Z} / 4 \mathbb{Z}.
\end{aligned}
\end{equation}


\subsection{Condensation of the Drinfeld center of $\mathbb{Z}/4\mathbb Z\times\mathbb{Z}/4\mathbb Z+16$}

Let $\mathcal{C}$ be the near-group category of type $\mathbb{Z}/4\mathbb Z\times\mathbb{Z}/4\mathbb Z+16$, then the Drinfeld center of $\mathcal C$, denoted by $\mathcal Z(\mathcal C)$, is modular. By analyzing the twists of the pointed subcategory of $\mathcal{Z}(\mathcal{C})$ in Equation  $\left(\ref{eq:pointedST}\right)$, we find that $\mathcal{Z}(\mathcal{C})$ contains a Tannakian fusion subcategory $\mathcal{E} = \operatorname{Rep}(\mathbb{Z}/2\mathbb{Z} \times \mathbb{Z}/4\mathbb{Z})$. Furthermore, we have the dimensions  $\dim(\mathcal E_{\mathcal{Z}(\mathcal{C})}')=640(9+4\sqrt{5})$ and $\dim(\mathcal{Z}(\mathcal{C})_G^0)=80(2+\sqrt{5})^2$, where $G={\mathbb Z/2\mathbb Z\times\mathbb Z/4\mathbb {Z}}$. 

Denote $Y_i$ and $Z_j$, $W_k$ as the simple objects of $\mathcal{Z}(\mathcal{C})$ with dimensions $\chi_{80}^9$, $\chi_{80}^{10}$, $\chi_{80}^8$, respectively. For simplicity of the indices,  we take the ordering of indices as following: $Y_i$'s are indexed by elements in $\mathbb{Z} / 4 \mathbb{Z} \times \mathbb{Z} / 4 \mathbb{Z}$ with order $(0,0),\cdots, (0,3),(1,0),\cdots,(3,3) $,   $Z_j$'s are indexed by $120$ pairs of elements $\{(g_1,g_2), (h_1, h_2)\}$ in $\mathbb{Z} / 4 \mathbb{Z} \times \mathbb{Z} / 4 \mathbb{Z}$ with order 
\begin{align*}\{(0,0), (0,1)\},\cdots,\{(0,1),(0,2)\},\cdots, \{(3,2),(3,3)\}
\end{align*}
and $W_k$ are indexed with the ordering in Table \ref{table:solZ4Z4} in Appendix \ref{Appendix}. Then $\operatorname{rank}(\mathcal{E}_{\mathcal{Z}(\mathcal{C})}')=44$, moreover, we have 
\begin{equation}
\mathcal{O}\left(\mathcal{E}'_{\mathcal{Z}(\mathcal{C})}\right) = \mathcal{O}(\mathcal{E}) \mathrel{\displaystyle \cup}
\left\{ 
\begin{aligned}
    & Y_1, Y_3, Y_6, Y_8, Y_9, Y_{11}, Y_{14}, Y_{16}, \\
    & Z_{10}, Z_{17}, Z_{23}, Z_{35}, Z_{50}, Z_{56}, Z_{62}, Z_{75}, Z_{83},  Z_{90}, Z_{101}, Z_{116}, \\
    & W_1, W_2, W_3, W_4, W_5, W_6, W_7, W_8, \\
    & W_{95}, W_{96}, W_{97}, W_{98}, W_{99}, W_{100}, W_{101}, W_{102}
\end{aligned}
\right\}.
\end{equation}


 Next we determine the set of simple objects of the condensation $\mathcal{D}:=\mathcal{Z}(\mathcal{C})_G^0$ of $\mathcal{Z}(\mathcal{C})$ by $\mathcal{E}$. Let $A:=\operatorname{Fun}(\mathbb  Z/4\mathbb Z\mathbb \times  \mathbb Z/2\mathbb Z)$, we list the action of $A$ on simple objects of $\mathcal{E}_{\mathcal{Z}(\mathcal{C})}'$ in Table \ref{condensationZ4*Z4}.

\noindent Hence we obtain the simple objects of $\mathcal{D}$:
\begin{align*}
 \{F(\mathbf{1}), F(Y_1), (Z_{10})_1,(Z_{10})_2,(Z_{17})_1,(Z_{17})_2,(Z_{23})_1,(Z_{23})_2,F(W_1), F(W_5)\}.
\end{align*}
Next, we determine the modular data of $\mathcal{D}$. Since \begin{align*}\theta_{Y_1}=\theta_{Z_{10}}=1, \theta_{Z_{17}}=\zeta_4, \theta_{Z_{23}}=\zeta_4^3,  \theta_{W_1}=\zeta_5^2,
\theta_{W_5}=\zeta_5^3,
\end{align*}it follows from \cite{KirillovOstrik2001} that the  $T$-matrix of  $\mathcal{D}$  is given by:
\begin{align*}
 \operatorname{diag}\left(1,1,1,1,\zeta_4,\zeta_4,\zeta_4^3,\zeta_4^3,\zeta_5^2,\zeta_5^3\right).
\end{align*}
The remaining task is to determine the  $S$-matrix of  $\mathcal{D}$.

\begin{table}[ht]
    \centering
    \begin{tabular}{|c|c|c|c}
 \hline Dim & $A \otimes-$ &  \\
 \hline  $\chi_{80}^9$   & $A\otimes Y_i=Y_1\oplus Y_3\oplus Y_6\oplus Y_8\oplus Y_9\oplus Y_{11}\oplus Y_{14}\oplus  Y_{16}$ & $i=1,3,6,8,9,11,14,16$   \\
 \hline  $\chi_{80}^{10}$ & $A\otimes Z_j=2(Z_{10}\oplus Z_{35}\oplus Z_{75}\oplus Z_{90})$& $j=10,35,75,90$\\ 
  \hline  $\chi_{80}^{10}$ &$A\otimes Z_j=2(Z_{17}\oplus Z_{56}\oplus Z_{101}\oplus Z_{116})$& $j=17,56,101,116$\\ 
   \hline  $\chi_{80}^{10}$ &$A\otimes Z_j=2(Z_{23}\oplus Z_{50}\oplus Z_{62}\oplus Z_{83})$& $j=23,50,62,83$\\  
 \hline $\chi_{80}^8$ &  $A\otimes W_k=\oplus_{k=1}^4(W_k\oplus W_{k+94})$& $1\leq k\leq 4$, $95\leq k\leq 98$
  \\
  \hline $\chi_{80}^8$ &  $A\otimes W_k=\oplus_{k=5}^8(W_k\oplus W_{k+94})$& $5\leq k\leq 8$, $99\leq k\leq 102$
  \\ 
 \hline 
    \end{tabular}
    \caption{Action of \'{e}tale algebra A}
    \label{condensationZ4*Z4}
\end{table}
\begin{lemma}
$S_{F(Y_1),F(Y_1)}=1$, $S_{F(Y_1), F(W_j)}=-\chi_{80}^8$,
$S_{F(W_1),F(W_1)}=S_{F(W_5),F(W_5)}=-\chi_{20}^6$, $S_{F(W_1),F(W_5)}=\chi_{180}^{14}$.
\end{lemma} 
\begin{proof}
By \cite[Lemma 4.2]{KirillovOstrik2001}, we have $S_{F(V_1),F(V_2)}=\frac{1}{64}S_{A\otimes V_1, A\otimes V_2}$, for any  objects $V_1,V_2\in \{W_1,W_5,Y_1\}$. Notice that the \'{e}tale algebra $A$ acts freely on  these simple objects, hence \begin{align*}S_{F(V_1),F(V_2)}=\frac{1}{64}\sum_{g,h\in\mathbb{Z}/4\mathbb Z\times \mathbb{Z}/2\mathbb Z}S_{g\otimes V_1,h\otimes V_2}=S_{V_1,V_2}.\end{align*} Thus the result follows. 
\end{proof}

\begin{lemma}
For all $V\in\mathcal{O}(\mathcal{D})$, $S_{F(Y_1),V}=\chi_{80}^9\, \sigma(\dim(V))$,   where $\sigma(\sqrt{5})=-\sqrt{5}$. Additionally,   $S_{F(W_1),(Z_k)_j}=S_{F(W_5),(Z_k)_j}=0$.
\end{lemma}

\begin{proof}
   Let $\sigma\in \operatorname{Gal}(\mathbb Q(T)/\mathbb Q)$ be the Galois automorphism  such that $\sigma(\sqrt{5})=-\sqrt{5}$. Then $\sigma(\dim(\mathcal{D}
   ))=\frac{\dim(\mathcal{D})} {\dim(\hat{\sigma}(F(\mathbf{1})))^2}$. This implies $\dim(\hat{\sigma}(F(\mathbf{1})))=\chi_{80}^9$, so  $\hat{\sigma}(F(\mathbf{1}))=F(Y_1)$ and consequently, $S_{F(Y_1),V}=\chi_{80}^9\,\sigma(\dim(V))$.  
 In particular, $S_{(Z_k)_j,F(Y_1)}=\chi_{20}^5$.

As $\theta_{W_1}=\zeta_5^2 \text { and } \theta_{W_5}=\zeta_5^3$, $F(W_1)$ and $F(W_5)$ are in the same orbit under the action of $\sigma$. Since $S_{F(W_i),F(W_j)}$ are real algebraic integers, $F(W_1)$ and $F(W_5)$ are self-dual. Therefore $S_{(Z_k)_j,F(W_l)}$ are also real algebraic integers, where $l=1,5$. Since $\dim(\mathcal{D})=2(\chi_{80}^8)^2+(\chi_{180}^{14})^2+(\chi_{20}^6)^2$, we conclude that  $S_{(Z_k)_j,F(W_l)}=0$ for $l=1,5$.
\end{proof}


Thus the $S$-matrix of $\mathcal{D}$ is of the form 
\begin{align}\label{eq:Swith6unknown}
    \left(\begin{array}{cccccccccc}
         1& \chi_{80}^9&\chi_{20}^5&\chi_{20}^5&
        \chi_{20}^5 &\chi_{20}^5& \chi_{20}^5 &\chi_{20}^5&\chi_{80}^8&\chi_{80}^8\\
        \chi_{80}^9 &1&\chi_{20}^5&\chi_{20}^5&
        \chi_{20}^5 &\chi_{20}^5& \chi_{20}^5 &\chi_{20}^5&
        -\chi_{80}^8&-\chi_{80}^8\\
       \chi_{20}^5&\chi_{20}^5&&&&&&&0&0\\
       \chi_{20}^5&\chi_{20}^5&&&&&&&0&0\\
        \chi_{20}^5 &\chi_{20}^5&&&***&&&&0&0\\
        \chi_{20}^5&\chi_{20}^5&&&&&&&0&0\\
        \chi_{20}^5 &\chi_{20}^5&&&&&&&0&0\\
        \chi_{20}^5&\chi_{20}^5&&&&&&&0&0\\
        \chi_{80}^8&-\chi_{80}^8&0&0&0&0&0&0&
        -\chi_{20}^6&\chi_{180}^{14}\\
        \chi_{80}^8&-\chi_{80}^8&0&0&0&0&0&0&\chi_{180}^{14}&-\chi_{20}^6\\
    \end{array}\right)
\end{align}

To determine the remaining entries of the $S$-matrix of $\mathcal{D}$ in $(\ref{eq:Swith6unknown})$, we apply the method introduced in the reconstruction program in \cite{NRWWrank6,ng2023classification}. Specifically, we begin by identifying all possible representation types of the associated modular congruence representation of $\mathcal{D}$. This involves decomposing the representation into  a direct sum of irreducible components. Next,  we reconstruct the modular data of $\mathcal{D}$ from the representation of $\operatorname{SL}(2,\mathbb Z)$, leveraging properties of modular categories, such as Galois symmetry, the Verlinde formula, and other structural constraints.

Let $\rho_\mathcal{D}$ be the $\operatorname{SL}(2,\mathbb Z)$  modular congruence representation of $\mathcal{D}$, then it follows from the spectrum of $\rho_\mathcal{D}(\mathfrak{t})$ that   $\rho_\mathcal{D}$ contains a unique direct summand $\rho_1$ of dimension $3$ of level $5$, where
$$\rho_1(\mathfrak{s})=\frac{-1}{\sqrt{5}}\left(\begin{array}{ccc}
     1& \sqrt{2} &\sqrt{2}\\
     \sqrt{2}&2\cos\frac{2\pi}{5}&2\cos\frac{6\pi}{5}\\
     \sqrt{2}&2\cos\frac{6\pi}{5}&2\cos\frac{2\pi}{5}\\
\end{array}\right) \text{ and } \rho_1(\mathfrak{t})=\operatorname{diag}(1,\zeta_5^2,\zeta_5^3).$$
Assume $\rho_\mathcal{D}\cong \rho_1\oplus \rho'$, then $\rho'$ is a direct sum of irreducible representations with level $4$, whose $\mathfrak{t}$-spectrum belongs to 
$\{1,\zeta_4,\zeta_4^3\}$. Notice that these irreducible direct summands of level $4$ are  $\rho_2,\rho_3,\rho_4,\rho_5$ \cite{Nobs1,Nobs2,plavnik2023modular}, where
\begin{align*}
    \rho_2(\mathfrak{s})=\frac{1}{2}\left(\begin{array}{ccc}
        -1 & 1&\sqrt{2} \\
      1 & -1&\sqrt{2} \\
       \sqrt{2} & -\sqrt{2}&0 \\  
    \end{array}\right), \, \rho_2(\mathfrak{t})=\operatorname{diag}(\zeta_4,\zeta_4^3,1),
\end{align*}
\begin{align*}
    \rho_3(\mathfrak{s})=\frac{1}{2}\left(\begin{array}{cc}
        \zeta_4 & \sqrt{3}\zeta_4 \\
     \sqrt{3}\zeta_4 & -\zeta_4 \\  
    \end{array}\right), \, \rho_3(\mathfrak{t})=\operatorname{diag}(\zeta_4,\zeta_4^3),\\
\rho_4(\mathfrak{s})=\zeta_4^3,\, \rho_4(\mathfrak{t})=\zeta_4,\, \rho_5(\mathfrak{s})=\zeta_4,\, \rho_5(\mathfrak{t})=\zeta_4^3.
\end{align*}
\begin{lemma}\label{level4case}
    Let $\rho_\mathcal{D}=\rho_1\oplus \rho'$ be the congruence $\operatorname{SL}(2,\mathbb Z)$ representation associated $\mathcal{D}$. The  $\rho_\mathcal{D}$ is  equivalent to one of the following decompositions:
    \begin{align*}
       \rho_1\oplus2\rho_2\oplus\rho_0,  \quad \rho_1\oplus\rho_2\oplus\rho_3\oplus2\rho_0,  \quad \rho_1\oplus\rho_2\oplus\rho_4\oplus\rho_5\oplus 2\rho_0.
    \end{align*}
\end{lemma}

\begin{proof}
By comparing the $\mathfrak{t}$-spectrum of $\rho'(\mathfrak{t})$ with that of $\rho_\mathcal{D}(\mathfrak{t}
)$, a direct computation shows $\rho'$ is equivalent to one of the following
$$\begin{aligned}
    &2\rho_2 \oplus \rho_0, \quad \rho_2 \oplus \rho_3 \oplus 2\rho_0, \quad \rho_2 \oplus \rho_4 \oplus \rho_5 \oplus 2\rho_0, \\ 
    &2\rho_3 \oplus 3\rho_0, \quad \rho_3 \oplus \rho_4 \oplus \rho_5 \oplus 3\rho_0, \quad 2\rho_4 \oplus 2\rho_5 \oplus 3\rho_0.
\end{aligned}$$
If $\rho_\mathcal{D}\cong(\rho_1\oplus 3\rho_0)\oplus \rho''$ for $\rho''=2\rho_3,\rho_3\oplus\rho_4\oplus\rho_5$, or $2\rho_4\oplus2\rho_5$, then $\rho_\mathcal{D}$ is a direct sum of two sub-representations with disjoint $\mathfrak{t}$-spectrums. However,  it is not possible  by   \cite[Lemma 3.18]{BNRWbyrank2016}. This completes the proof of the lemma.
\end{proof}
  
 \begin{proposition} \label{case(3,2,1,1)}
Let $\rho_\mathcal{D}$ be the congruence 
 $\operatorname{SL}(2,\mathbb Z)$ representation associated with $\mathcal{D}$. Then $\rho_\mathcal{D}\ncong\rho_1\oplus\rho_2\oplus\rho_3\oplus2\rho_0$.
 \end{proposition}
 \begin{proof}
     Assume $\rho_\mathcal{D}\cong \tilde{\rho}:=\rho_1\oplus\rho_2\oplus\rho_3\oplus2\rho_0$. Notice that 
     \begin{align*}
         \tilde{\rho}(\mathfrak{s})=\left(
         \begin{array}{cccccccccc}
             1&0&0&0&0&0&0&0 &  0&0\\
              0&1&0&0&0&0&0&0&0&0\\
              0&0&-\frac{1}{\sqrt{5}}&0&0&0&0&0 &-\frac{\sqrt{2}}{\sqrt{5}}&-\frac{\sqrt{2}}{\sqrt{5}}\\
               0&0&0&0&\frac{1}{\sqrt{2}}&0&\frac{1}{\sqrt{2}}&0&0&0\\
                0&0&0&\frac{1}{\sqrt{2}}&-\frac{1}{2}&0&\frac{1}{2}&0&0&0\\
                0&0&0&0&0&\frac{\zeta_4}{2}&0&\frac{\sqrt{3}\zeta_4}{2}&0&0\\
              0&0&0&\frac{1}{\sqrt{2}}&\frac{1}{2}&0&-\frac{1}{2}&0&0&0\\
               0&0&0&0&0&\frac{\sqrt{3}\zeta_4}{2}&0&-\frac{-\zeta_4}{2}&0&0\\               
               0&0 &-\frac{\sqrt{2}}{\sqrt{5}}&0&0&0&0 &  0&-\frac{2\cos\frac{2\pi}{5}}{\sqrt{5}}&-\frac{2\cos\frac{6\pi}{5}}{\sqrt{5}}\\
               0&0 &-\frac{\sqrt{2}}{\sqrt{5}}&0&0&0&0 &  0&-\frac{2\cos\frac{6\pi}{5}}{\sqrt{5}}&-\frac{2\cos\frac{2\pi}{5}}{\sqrt{5}}
         \end{array}\right).
     \end{align*}
     It follows from \cite[Theorem 3.4]{NRWWrank6} that there is a real orthogonal matrix $U=(a_{ij})$   such that $\rho_\mathcal{D}(\mathfrak{s})=U\tilde{\rho}(\mathfrak{s})U^t$ and $U\rho_\mathcal{D}(\mathfrak{t})=\rho_\mathcal{D}(\mathfrak{t})U$.  Then $U=\left(\begin{array}{cccc}
          A_1&&&\\
          &A_2&&\\
          &&A_3&\\
          &&&A_4 \\
     \end{array}\right)$ and $A_4=\left(\begin{array}{cc}
          \mu_1&  \\
          & \mu_2
     \end{array}\right)$, $\mu_1,\mu_2\in\{\pm1\}$.

     Hence, we obtain  $\rho_\mathcal{D}(\mathfrak{s})=\left(\begin{array}{cc} 
       B_1&B_2\\
       B_2^t&B_3
     \end{array}\right)$, where\\  
\small{$$B_1=\begin{pmatrix}
1 - \frac{a_{13}^2 + 1}{\sqrt{5}} - a_{14}^2 & -a_{13} a_{23} \frac{\sqrt{5} + 1}{\sqrt{5}} - a_{14} a_{24} & -a_{13} a_{33} \frac{\sqrt{5} + 1}{\sqrt{5}} - a_{14} a_{34} & -a_{13} a_{43} \frac{\sqrt{5} + 1}{\sqrt{5}} - a_{14} a_{44} \\
-a_{13} a_{23} \frac{\sqrt{5} + 1}{\sqrt{5}} - a_{14} a_{24} & 1 - \frac{a_{23}^2 + 1}{\sqrt{5}} - a_{24}^2 & -a_{23} a_{33} \frac{\sqrt{5} + 1}{\sqrt{5}} - a_{24} a_{34} & -a_{23} a_{43} \frac{\sqrt{5} + 1}{\sqrt{5}} - a_{24} a_{44} \\
-a_{13} a_{33} \frac{\sqrt{5} + 1}{\sqrt{5}} - a_{14} a_{24} & -a_{23} a_{33} \frac{\sqrt{5} + 1}{\sqrt{5}} - a_{24} a_{34} & 1 - \frac{a_{33}^2 + 1}{\sqrt{5}} - a_{34}^2 & -a_{33} a_{43} \frac{\sqrt{5} + 1}{\sqrt{5}} - a_{34} a_{44} \\
-a_{13} a_{43} \frac{\sqrt{5} + 1}{\sqrt{5}} - a_{14} a_{44} & -a_{23} a_{43} \frac{\sqrt{5} + 1}{\sqrt{5}} - a_{24} a_{44} & -a_{33} a_{43} \frac{\sqrt{5} + 1}{\sqrt{5}} - a_{34} a_{44} & 1 - \frac{a_{43}^2 + 1}{\sqrt{5}} - a_{44}^2
\end{pmatrix}$$}

\small{$$B_2 = \begin{pmatrix}
\frac{a_{14} a_{55}}{\sqrt{2}} & \frac{a_{14} a_{65}}{\sqrt{2}} & \frac{a_{14} a_{77}}{\sqrt{2}} & \frac{a_{14} a_{87}}{\sqrt{2}} & -\frac{\sqrt{2} a_{13} \mu_1}{\sqrt{5}} & -\frac{\sqrt{2} a_{13} \mu_2}{\sqrt{5}} \\
\frac{a_{24} a_{55}}{\sqrt{2}} & \frac{a_{24} a_{65}}{\sqrt{2}} & \frac{a_{24} a_{77}}{\sqrt{2}} & \frac{a_{24} a_{87}}{\sqrt{2}} & -\frac{\sqrt{2} a_{23} \mu_1}{\sqrt{5}} & -\frac{\sqrt{2} a_{23} \mu_2}{\sqrt{5}} \\
\frac{a_{34} a_{55}}{\sqrt{2}} & \frac{a_{34} a_{65}}{\sqrt{2}} & \frac{a_{34} a_{77}}{\sqrt{2}} & \frac{a_{34} a_{87}}{\sqrt{2}} & -\frac{\sqrt{2} a_{33} \mu_1}{\sqrt{5}} & -\frac{\sqrt{2} a_{33} \mu_2}{\sqrt{5}} \\
\frac{a_{44} a_{55}}{\sqrt{2}} & \frac{a_{44} a_{65}}{\sqrt{2}} & \frac{a_{44} a_{77}}{\sqrt{2}} & \frac{a_{44} a_{87}}{\sqrt{2}} & -\frac{\sqrt{2} a_{43} \mu_1}{\sqrt{5}} & -\frac{\sqrt{2} a_{43} \mu_2}{\sqrt{5}} \\
\end{pmatrix}$$}

$$B_3=\frac{1}{2}\left(\begin{array}{cc} B_3^{(1,2,3,4)} &\mathbf 0_{4\times 2}\\\mathbf 0_{2\times 4}&B_3^{(5,6)}\end{array}\right), \text{ with } B_3^{(5,6)} = \begin{pmatrix}
\frac{-4 \cos \frac{2\pi}{5}}{\sqrt{5}} & \frac{-4 \mu_1 \mu_2 \cos \frac{6\pi}{5}}{\sqrt{5}} \\
\frac{-4 \mu_1 \mu_2 \cos \frac{6\pi}{5}}{\sqrt{5}} & \frac{-4 \cos \frac{2\pi}{5}}{\sqrt{5}}
\end{pmatrix}, \text{ and } B_3^{(1,2,3, 4)} = $$

\footnotesize{$\begin{pmatrix}
-a_{55}^2 + \zeta_4 a_{56}^2 & -a_{55} a_{65} + \zeta_4 a_{56} a_{66} & a_{55} a_{77} + \sqrt{3} \zeta_4 a_{56} a_{78} & a_{55} a_{65} + \sqrt{3} \zeta_4 a_{56} a_{88} \\
-a_{55} a_{65} + \zeta_4 a_{56} a_{66} & -a_{65}^2 + \zeta_4 a_{66}^2 & a_{65} a_{77} + \sqrt{3} \zeta_4 a_{66} a_{78} & a_{65} a_{87} + \sqrt{3} \zeta_4 a_{66} a_{88} \\
a_{55} a_{77} + \sqrt{3} \zeta_4 a_{56} a_{78} & -a_{65} a_{77} + \zeta_4 a_{66} a_{78} & -\left(a_{77}^2 + \zeta_4 a_{78}^2\right) & -\left(a_{77} a_{87} + \sqrt{3} \zeta_4 a_{78} a_{88}\right) \\
a_{55} a_{87} + \zeta_4 a_{56} a_{88} & a_{65} a_{87} + \sqrt{3} \zeta_4 a_{66} a_{88} & -\left(a_{77} a_{87} + \sqrt{3} \zeta_4 a_{78} a_{88}\right) & -\left(a_{87}^2 + \zeta_4 a_{88}^2\right)
\end{pmatrix}$.}\\

  \normalsize Without loss of generality, assume the first row and the second row are determined by the unit object and its Galois conjugate. Hence we have \begin{align*}a_{13}\mu_j=a_{23}\mu_j=-\frac{1}{\sqrt{2}}, a_{33}=a_{43}=0, a_{13}a_{23}=-\frac{1}{2}, a_{14}=a_{24},a_{14}^2=\frac{1}{4}.
  \end{align*}
  Let $a_{13}=\frac{1}{\sqrt{2}}$ (the other case is similar), $\mu_1=\mu_2=-1$, $a_{14}=\frac{1}{2}$.   From coefficients in  the first row and second row of the  $S$-matrix, we have \begin{align*}
a_{14}a_{55}=a_{14}a_{65}=a_{14}a_{77}=a_{14}a_{87}=\frac{1}{2\sqrt{2}},
\end{align*}
which implies $a_{55}=a_{65}=a_{77}=a_{87}=\frac{1}{\sqrt{2}}$. The  orthogonality of matrices $A_2$ and $A_3$ implies $a_{56}a_{66}=-\frac{1}{2}=a_{78}a_{88}$. Then the $S$-matrix $S=\sqrt{\dim(\mathcal{D})}\,\rho_\mathcal{D} 
 (\mathfrak{s})=$ \begin{align*}
    \left(\begin{array}{cccccccccc}
         1& \chi_{80}^9&\chi_{20}^5&\chi_{20}^5&
        \chi_{20}^5 &\chi_{20}^5& \chi_{20}^5 &\chi_{20}^5&\chi_{80}^8&\chi_{80}^8\\
        \chi_{80}^9 &1&\chi_{20}^5&\chi_{20}^5&
        \chi_{20}^5 &\chi_{20}^5& \chi_{20}^5 &\chi_{20}^5&
        -\chi_{80}^8&-\chi_{80}^8\\
       \chi_{20}^5&\chi_{20}^5&3\chi_{20}^5&
       -\chi_{20}^5&-\chi_{20}^5&-\chi_{20}^5&-\chi_{20}^5&-\chi_{20}^5&0&0\\
       \chi_{20}^5&\chi_{20}^5&-\chi_{20}^5&
       3\chi_{20}^5&-\chi_{20}^5&-\chi_{20}^5&-\chi_{20}^5&-\chi_{20}^5&0&0\\
        \chi_{20}^5 &\chi_{20}^5&-\chi_{20}^5&-\chi_{20}^5&c_1\chi_{20}^5&\overline{c_1}\chi_{20}^5&c_2\chi_{20}^5& \overline{c_2}\chi_{20}^5&0&0\\
        \chi_{20}^5&\chi_{20}^5&-\chi_{20}^5&-\chi_{20}^5&\overline{c_1}\chi_{20}^5&c_1\chi_{20}^5&\overline{c_2}\chi_{20}^5& c_2\chi_{20}^5&0&0\\
        \chi_{20}^5 &\chi_{20}^5&-\chi_{20}^5&-\chi_{20}^5&c_2\chi_{20}^5& \overline{c_2}\chi_{20}^5&c_1\chi_{20}^5&\overline{c_1}\chi_{20}^5&0&0\\
        \chi_{20}^5&\chi_{20}^5&-\chi_{20}^5&-\chi_{20}^5&\overline{c_2}\chi_{20}^5& c_2\chi_{20}^5&\overline{c_1}\chi_{20}^5&c_1\chi_{20}^5&0&0\\
        \chi_{80}^8&-\chi_{80}^8&0&0&0&0&0&0&
        -\chi_{20}^6&\chi_{180}^{14}\\
        \chi_{80}^8&-\chi_{80}^8&0&0&0&0&0&0&\chi_{180}^{14}&-\chi_{20}^6\\
    \end{array}\right),
\end{align*}
where $c_1:=(-1+\zeta_4)$ and $c_2=(-1+\sqrt{3}\zeta_4)$.
However,  the forth and fifth rows violate the orthogonality of $S$-matrix. Therefore, $\rho_\mathcal{D}\ncong\rho_1\oplus\rho_2\oplus\rho_3\oplus2\rho_0$.  
 \end{proof}
  \begin{proposition}
      \label{case(3,3,3,1)}
Let $\rho_\mathcal{D}$ be the congruence $\operatorname{SL}(2,\mathbb Z)$ representation of $\mathcal{D}$. Then 
 $\rho_\mathcal{D}\ncong \rho_1\oplus2\rho_2\oplus\rho_0$.
 \end{proposition}
\begin{proof}
This proposition can be proved by using the same argument as Proposition \ref{case(3,2,1,1)}, we omit the details here.
\end{proof}
  \begin{theorem}\label{case(3,3,1,1,1,1)}
    Let $\rho_\mathcal{D}$ be the congruence $\operatorname{SL}(2,\mathbb Z)$ representation of $\mathcal{D}$. Then   $\rho_\mathcal{D}\cong \rho_1\oplus\rho_2\oplus\rho_4\oplus\rho_5\oplus 2\rho_0$. And the $S$-matrix of $\mathcal{D}$ is

 \begin{align*}
    \left(\begin{array}{cccccccccc}
         1& \chi_{80}^9&\chi_{20}^5&\chi_{20}^5&
        \chi_{20}^5 &\chi_{20}^5& \chi_{20}^5 &\chi_{20}^5&\chi_{80}^8&\chi_{80}^8\\
        \chi_{80}^9 &1&\chi_{20}^5&\chi_{20}^5&
        \chi_{20}^5 &\chi_{20}^5& \chi_{20}^5 &\chi_{20}^5&
        -\chi_{80}^8&-\chi_{80}^8\\
       \chi_{20}^5&\chi_{20}^5&3\chi_{20}^5&
       -\chi_{20}^5&-\chi_{20}^5&-\chi_{20}^5&-\chi_{20}^5&-\chi_{20}^5&0&0\\
       \chi_{20}^5&\chi_{20}^5&-\chi_{20}^5&
       3\chi_{20}^5&-\chi_{20}^5&-\chi_{20}^5&-\chi_{20}^5&-\chi_{20}^5&0&0\\
        \chi_{20}^5 &\chi_{20}^5&-\chi_{20}^5&-\chi_{20}^5&s\chi_{20}^5&\overline{s}\chi_{20}^5&\chi_{20}^5& \chi_{20}^5&0&0\\
        \chi_{20}^5&\chi_{20}^5&-\chi_{20}^5&-\chi_{20}^5&\overline{s}\chi_{20}^5&s\chi_{20}^5&\chi_{20}^5& \chi_{20}^5&0&0\\
        \chi_{20}^5 &\chi_{20}^5&-\chi_{20}^5&-\chi_{20}^5&\chi_{20}^5& \chi_{20}^5&s\chi_{20}^5&\overline{s}\chi_{20}^5&0&0\\
        \chi_{20}^5&\chi_{20}^5&-\chi_{20}^5&-\chi_{20}^5&\chi_{20}^5& \chi_{20}^5&\overline{s}\chi_{20}^5&s\chi_{20}^5&0&0\\
        \chi_{80}^8&-\chi_{80}^8&0&0&0&0&0&0&
        -\chi_{20}^6&\chi_{180}^{14}\\
        \chi_{80}^8&-\chi_{80}^8&0&0&0&0&0&0&\chi_{180}^{14}&-\chi_{20}^6\\
    \end{array}\right),
\end{align*}
where $s:=(-1+2\zeta_4)$.
\end{theorem}

\begin{proof}
    By Lemma \ref{level4case}, Proposition \ref{case(3,2,1,1)} and Proposition \ref{case(3,3,3,1)}, we know $\rho_\mathcal{D}\cong\rho_1\oplus\rho_2\oplus\rho_4\oplus\rho_5\oplus 2\rho_0$. Let $\tilde{\rho}:=\rho_1\oplus\rho_2\oplus\rho_4\oplus\rho_5\oplus 2\rho_0$, then 
    
    $$\tilde{\rho}(\mathfrak{s})=\left(
         \begin{array}{cccccccccc}
             1&0&0&0&0&0&0&0 &  0&0\\
              0&1&0&0&0&0&0&0&0&0\\
              0&0&-\frac{1}{\sqrt{5}}&0&0&0&0&0 &-\frac{\sqrt{2}}{\sqrt{5}}&-\frac{\sqrt{2}}{\sqrt{5}}\\
               0&0&0&0&\frac{1}{\sqrt{2}}&0&\frac{1}{\sqrt{2}}&0&0&0\\
                0&0&0&\frac{1}{\sqrt{2}}&-\frac{1}{2}&0&\frac{1}{2}&0&0&0\\
                0&0&0&0&0&-\zeta_4&0&0&0&0\\
              0&0&0&\frac{1}{\sqrt{2}}&\frac{1}{2}&0&-\frac{1}{2}&0&0&0\\
               0&0&0&0&0&0&0&\zeta_4&0&0\\          
               0&0 &-\frac{\sqrt{2}}{\sqrt{5}}&0&0&0&0 &  0&-\frac{2\cos\frac{2\pi}{5}}{\sqrt{5}}&-\frac{2\cos\frac{6\pi}{5}}{\sqrt{5}}\\
               0&0 &-\frac{\sqrt{2}}{\sqrt{5}}&0&0&0&0 &  0&-\frac{2\cos\frac{6\pi}{5}}{\sqrt{5}}&-\frac{2\cos\frac{2\pi}{5}}{\sqrt{5}}
         \end{array}\right).$$

         Again, we obtain $\rho_\mathcal{D}(\mathfrak{s})=U\tilde{\rho}(\mathfrak{s})U^t=\left(\begin{array}{cc} 
       B_1&B_2\\
       B_2^t&\widetilde{B_3}
     \end{array}\right)$, where $$\widetilde{B_3}=\frac{1}{2}\left(\begin{array}{cc} \widetilde{B_3}^{(1,2,3,4)} &\mathbf 0_{4\times 2}\\\mathbf 0_{2\times 4}&B_3^{(5,6)}\end{array}\right),  \text{ with } \widetilde{B_3}^{(1,2,3, 4)} = $$

\footnotesize{$\begin{pmatrix}
-a_{55}^2 + 2\zeta_4 a_{56}^2 & -a_{55} a_{65} + 2\zeta_4 a_{56} a_{66} & 0 & 0 \\
-a_{55} a_{65} + 2\zeta_4 a_{56} a_{66} & -a_{65}^2 + 2\zeta_4 a_{66}^2 &0 & 0 \\
0 & 0 & -\left(a_{77}^2 + 2\zeta_4 a_{78}^2\right) & -\left(a_{77} a_{87} +2 \zeta_4 a_{78} a_{88}\right) \\
0 & 0 & -\left(a_{77} a_{87} + 2 \zeta_4 a_{78} a_{88}\right) & -\left(a_{87}^2 + 2\zeta_4 a_{88}^2\right)
\end{pmatrix}$.} \\

\normalsize{Applying the same method as in Proposition \ref{case(3,2,1,1)}, we obtain the same relations for the orthogonal matrix $U$.  Consequently, the desired $S$-matrix follows immediately as  $S=\sqrt{\dim(\mathcal{D})}\, \rho_\mathcal{D}(\mathfrak{s})$. }     
\end{proof}

\begin{proposition}\label{rank4category}
    $\mathcal{D}$ is the Drinfeld center of fusion category $\mathcal{A}$ with fusion rule
    \begin{align*}
        Y_1\otimes  Y_1=\mathbf{1}\oplus 2Y_1\oplus 2Y_2,\quad g\otimes Y_1=Y_2,\quad g\otimes g=\mathbf{1}.
    \end{align*}
\end{proposition}
\begin{proof}
  It follows from Proposition \ref{equivariantization} that $\mathcal{D}$ is braided equivalent to the Drinfeld center of some fusion category $\mathcal{A}$.  Next, we need to determine the fusion rules of $\mathcal{A}$. Let $\mathcal{C}$ be the near-group category of type ${\mathbb{Z}/4\mathbb{Z}\times\mathbb{Z}/4\mathbb{Z}}+16$, let  $F_A$ be the free functor from $\mathcal{C}$ to $\mathcal{C}_A$. 
  
  It is obvious that $\mathcal{A}$ contains a non-trivial invertible object $g$ such that $g\otimes g=\mathbf{1}$. Since $\dim(\mathcal{A})=20+8\sqrt{5}$ and $\dim(Z)=4\chi_5^2$, where $Z$ is the unique non-invertible simple object of $\mathcal{C}$, it follows from \cite[Proposition 4.26]{Drinfeld2010} that the simple objects in $\mathcal A$ must have dimensions either $1$ or $\chi_5^2$.
  
  Since $F_A(h\otimes Z)=F_A(Z)$ for all $h\in {\mathbb{Z}/4\mathbb{Z}\times\mathbb{Z}/4\mathbb{Z}}$, we know $F_A(Z)=2Y_1\oplus 2Y_2$ and $g\otimes Y_1=Y_2$, where $Y_1,Y_2$ are non-invertible simple objects of $\mathcal{A}$ of dimension $\chi_5^2$. Therefore, $\mathcal{A}$ must be a fusion category of rank $4$ with a unique non-trivial invertible object $g$.  It then follows from  \cite[Theorem A]{Ediez2quadratic} that $\mathcal{A}$ is  self-dual  and $  Y_1\otimes  Y_1=\mathbf{1}\oplus 2Y_1\oplus 2Y_2$.
\end{proof}
\begin{remark}
    Note the Grothendieck ring $\mathcal{K}_0(\mathcal{A})$ is commutative, so the Lagrangian algebra $I(\mathbf{1})$ is multiplicity-free by \cite[Corollary 2.16]{ostrik2015pivotal}, where $I$ is the adjoint functor to the forgetful functor $\mathcal{D}\to\mathcal{A}$, hence we have $I(\mathbf{1})=F(\mathbf{1})\oplus F(Y_1)\oplus (Z_{10})_1\oplus (Z_{10})_2$, as the ribbon structure of $I(\mathbf{1})$ is trivial.
\end{remark}
\subsection{Near-group category $\mathbb Z/ 8\mathbb Z+8$}\label{sec:Z8}
There are $8$ inequivalent near-group categories of type $\mathbb Z/ 8\mathbb Z+8$ \cite{EvansGannon}. 
The data for $\mathbb{Z} /8\mathbb{Z}+8$  is listed in Table \ref{DataZ8}. The non-degenerate symmetric bi-character for a cyclic group $\mathbb{Z} / n\mathbb{Z}$ is of the form $\langle x, y\rangle=\exp (2 \pi i m x y / n)$, where $m \in \mathbb{Z}$ and $\operatorname{gcd}(m, n)=1$. When $n$ is even, the function $a$ is in the form of $a(x)=\varepsilon^x \exp \left(-\pi i m x^2 / n\right)$, where $\varepsilon= \pm 1$. By taking $m \in \mathbb{Z} / 2 n\mathbb{Z}$ and coprime to $n$, we can eliminate the potential factor of $-1$ . This $m$ is recorded in the second column of Table 1. Recall that $b(0)=-\frac{1}{d}$, $b(-x)=\overline{a(x) b(x)}$ and $b(x)=\frac{1}{\sqrt{n}} \exp (i j(x))$, where $j(x) \in[-\pi, \pi]$.  It suffices to list the values of $j(1), \cdots, j(4)$ to determine $b(x)$.

\begin{table}[ht]
\begin{tabular}{|l|c|l|l|}
\hline
id & \multicolumn{1}{l|}{$m$} & \multicolumn{1}{c|}{$c$} &\qquad \qquad$ j(1), j(2), j(3), j(4) $ \\ \hline
$J_8^1$  &        $-1$            &    $\zeta_{24}^{-1}$     &  $ 0.872276, -2.70426, 2.9768, 3.14159$ \\ $\overline{J_8^1}$ &        1    &  $\zeta_{24}$     &  $-0.872276, 2.70426, -2.9768, 3.14159 $\\ \hline
  $J_8^2$ &     $-1$                &   $\zeta_{24}^{-1}$        &   $-1.26498, 1.13347, -0.227903, 3.14159 $  \\
  $\overline{J_8^2}$ &                1      &     $\zeta_{24}$      &    $1.26498, -1.13347, 0.227903, 3.14159$  \\ \hline
    $J_8^3$ &    $-3$               &   $\zeta_{24}^{17}$       &  $ -2.46405, 3.07557, 0.491887, 0$ \\
  $\overline{J_8^3}$ &                3      &     $\zeta_{24}^{-17}$     &   $2.46405, -3.07557, -0.491887, 0$\\ \hline
    $J_8^4$ &   $-3$               &  $\zeta_{24}^{17}$      &  $1.28595, -1.50478, 1.47161, 0$ \\
  $\overline{J_8^4}$ &                3      &    $\zeta_{24}^{-17}$      &    $-1.28595, 1.50478, -1.47161, 0$\\ \hline

\end{tabular}

\caption{Near-group categories of type $\mathbb Z/8\mathbb Z+8$.}\label{DataZ8}
\end{table}

\subsubsection{The center of $\mathbb Z/8\mathbb Z+8$}

We focus on the near-group category  labeled by $J_8^1$ in Table \ref{DataZ8}. The solutions for the triples $\left(\omega_i, \tau_i, \xi_i\right)$ for Equations $\left(\ref{eq:half1}\right)$ -- $\left(\ref{eq:half4}\right)$ are listed in Table \ref{table:triplesforJ81} in Appendix \ref{Appendix}. Additional details regarding these solutions and the modular data of the center are provided in the Mathematica notebook \texttt{MDZ8.nb}, available in the arXiv source.

The center of $\mathbb Z/8\mathbb Z+8$ is a rank 88 modular category with the following simple objects:
\begin{itemize}
     \item eight invertible objects $X_g$, $g\in \mathbb Z/8\mathbb Z $;
    \item eight $\chi_{24}^5$-dimensional simple objects $Y_h$, $h\in  \mathbb Z/8\mathbb Z$;
    
    \item twenty-eight $2 \chi_6^3$-dimensional objects $Z_{k, l}, k, l \in \mathbb Z/8\mathbb Z$, and $k<l$; 
    \item forty-four $\chi_{24}^4$-dimensional objects $W_{\omega_i, \tau_i, \xi_i}$, where $\left(\omega_i, \tau_i, \xi_i\right)$ are solutions in Table \ref{table:triplesforJ81}.
\end{itemize}

The $T$-matrix is given as the following:

$T_{X_g}=e^{-\frac{1}{4}\pi i g^2}$, $T_{Y_h}=e^{-\frac{1}{4}\pi i h^2}$,  $T_{Z_{k,l}}=e^{-\frac{1}{4}\pi i kl}$, $g, h, k, l \in \mathbb Z/8\mathbb Z$, and


$T_{W_{\omega, \tau, \xi}}=\operatorname{diag}\left(
e^{-\frac{i \pi}{2}}, 
e^{\frac{5 i \pi }{6}}, 
e^{\frac{5 i \pi }{6}}, 
e^{-\frac{2 i \pi }{3}}, 
e^{-\frac{2 i \pi }{3}}, 
e^{\frac{13 i \pi }{16}}, 
e^{-\frac{3 i \pi }{16}}, 
e^{\frac{7 i \pi }{48}}, 
e^{\frac{7 i \pi }{48}}, 
e^{-\frac{41 i \pi }{48}}, 
e^{-\frac{41 i \pi }{48}}, 
e^{-\frac{3 i \pi }{4}}, 
e^{\frac{7 i \pi }{12}}, \right.$ \\$\left.
e^{\frac{7 i \pi }{12}}, 
e^{-\frac{11 i \pi }{12}}, 
e^{-\frac{11 i \pi }{12}}, 
e^{\frac{5 i \pi }{16}}, 
e^{-\frac{11 i \pi }{16}}, 
e^{\frac{31 i \pi }{48}}, 
e^{\frac{31 i \pi }{48}}, 
e^{-\frac{17 i \pi }{48}}, 
e^{-\frac{17 i \pi }{48}}, 
e^{\frac{i \pi}{2}}, 
e^{\frac{i \pi }{3}}, 
e^{\frac{i \pi }{3}}, 
e^{-\frac{i \pi }{6}}, 
e^{-\frac{i \pi }{6}}, 
e^{\frac{5 i \pi }{16}}, 
e^{-\frac{11 i \pi }{16}}, \right.$ \\$\left.
e^{\frac{31 i \pi }{48}}, 
e^{\frac{31 i \pi }{48}}, 
e^{-\frac{17 i \pi }{48}}, 
e^{-\frac{17 i \pi }{48}}, 
e^{-\frac{3 i \pi }{4}}, 
e^{\frac{7 i \pi }{12}}, 
e^{\frac{7 i \pi }{12}}, 
e^{-\frac{11 i \pi }{12}}, 
e^{-\frac{11 i \pi }{12}}, 
e^{\frac{13 i \pi }{16}}, 
e^{-\frac{3 i \pi }{16}}, 
e^{\frac{7 i \pi }{48}}, 
e^{\frac{7 i \pi }{48}}, 
e^{-\frac{41 i \pi }{48}}, 
e^{-\frac{41 i \pi }{48}}
\right)$.

And the $S$-matrix is
$$S= \left(\begin{array}{cccc}s_{X, X} & \chi_{24}^5 s_{X, X} & 2 \chi_6^3 s_{X, Z} & \chi_{24}^4 s_{X, W} \\ \chi_{24}^5 s_{X, X} & s_{X, X} & 2 \chi_6^3 s_{X, Z} & -\chi_{24}^4 s_{X, W} \\ 2 \chi_6^3 s_{X, Z}^T & 2 \chi_6^3 s_{X, Z}^T, & 2 \chi_6^3 s_{Z, Z} & \mathbf{0} \\ \chi_{24}^4 s_{X, W}^T & -\chi_{24}^4 s_{X, W}^T & \mathbf{0} & s_{W, W}\end{array}\right),$$
where 
$\lambda = 16 \sqrt{3 \chi_{24}^5}$, 
\begin{align*}
S_{X_g,X_{h}}= e^{\frac{1}{2}  \pi i gh}, \quad S_{X_g,Z_{k,l}}=e^{\frac{1}{4}  \pi i  g (k+l)},\quad S_{Z_{k,l}, Z_{k',l'}}= e^{\frac{1}{4}  \pi i  k l^\prime+\frac{1}{4}  \pi i k^\prime l}+e^{\frac{1}{4}  \pi i k k^\prime+\frac{1}{4}  \pi i  l l^\prime},
\end{align*}
where $g, h, k, l, k',l'\in \mathbb Z/8\mathbb Z$, $k < l$ and $k' < l'$;
and $s_{X, W}$ and $s_{W, W}$ can be computed from equations in \cite{EvansGannon} using the data from Table \ref{table:triplesforJ81}.

As $\theta_{X_4}=1$ for the invertible object $X_4$, we know $b=X_4$ is a boson. Therefore we can apply the boson condensation to $ \mathcal{Z}(\C)$ to obtain $\mathcal{Z}(\C)_{\mathbb Z/2\mathbb Z}^0$. By the Verlinde formula,  the fusion rules for tensoring with the boson  $b$  can be  computed. The fusion rules for tensoring with  $b$  and the objects centralized by  $b$  are summarized in Table  \ref{tableboson}.

\begin{table}[ht]
\begin{tabular}{|c|c|c|c|c|}
\hline Dim & $b \otimes-$ & & & in $\langle b\rangle^{\prime} ?$ \\
\hline 1 & $b X_g=X_{g+4}$ & $g \in \mathbb{Z} / 8 \mathbb{Z}$ & & Yes \\
\hline$\chi_{24}^5$ & $b Y_h=Y_{h+4}$ & $h \in \mathbb{Z} / 8 \mathbb{Z}$ & & Yes \\
\hline \multirow[t]{2}{*}{$2 \chi_6^3$} & $b Z_{(k, l)}=Z_{\left(k^{\prime}, l^{\prime}\right)}$ & \begin{tabular}{l}
$\begin{aligned}l^{\prime} & = \begin{cases}k+4, & k<4, l \geqslant 4 \\
l+4, & k<4, l<4 k^{\prime}\\
l-4, & k \geqslant 4\end{cases} \\
k^{\prime} & \equiv k+l-l^{\prime}( \bmod 8)\end{aligned}$
\end{tabular} & $k+l$ even & Yes \\
\cline{4-5} & & & $k+l$ odd & No \\
\hline
\multirow{5}{*}{$\chi_{24}^4$} & $b W_i=W_i$ & $i=1,12,23,34$ & & Yes \\
\cline{2-5}  & \multirow{2}{*}{$bW_i=W_{i+1}$}  & $i=2,4,13,15,24,26,35,37$ & & Yes \\
\cline{3-5}  & & $i=6,17,20,28,31,39$ & & No \\
\cline{2-5}  &    $b W_i=W_{i+2}$ & $i=8,9,41,42$ & & No \\
\cline{2-5}   & $b W_i=W_{i+3}$ & $i=19,30$ & & No \\
\cline{2-5}  
\hline
\end{tabular}\caption{Action of boson $b$}
\label{tableboson}
\end{table}
By examining the table, we find that $\mathcal{Z}(\C)_{\mathbb Z/2\mathbb Z}^0$ has rank $36$. 
The dimensions and twists of the simple objects are summarized in   Table \ref{TableRibboncondenZ8+8}.

\begin{table}
\centering
\begin{tabular}{|c|c|c|c|}
\hline
\textbf{dim} & \textbf{Objects} & \textbf{Twists} & \textbf{Number Count} \\
\hline
1 & $F(1), F(X_1), F(X_2), F(X_3)$ & $1, e^{-\frac{\pi i}{4}}, -1, e^{-\frac{\pi i}{4}}$ & 4 \\
\hline
$\chi^5_{24}$ & $F(Y_0), F(Y_1), F(Y_2), F(Y_3)$ & $1, e^{-\frac{\pi i}{4}}, -1, e^{-\frac{\pi i}{4}}$ & 4 \\
\hline
$2\chi^3_6$ & $F(Z_{0,2}), F(Z_{0,6}), F(Z_{1,3}), F(Z_{1,7})$ & $1, 1, e^{-\frac{3\pi i}{4}}, e^{\frac{\pi i}{4}}$ & 4 \\
\hline
$\chi^3_6$ & 
\begin{tabular}{@{}c@{}} 
$(Z_{0,4})_1, (Z_{0,4})_2, (Z_{1,5})_1, (Z_{1,5})_2$ \\ 
$(Z_{2,6})_1, (Z_{2,6})_2, (Z_{3,7})_1, (Z_{3,7})_2$
\end{tabular} 
& $1, 1, e^{\frac{3\pi i}{4}}, e^{\frac{3\pi i}{4}}, -1, -1, e^{\frac{3\pi i}{4}}, e^{\frac{3\pi i}{4}}$ & 8 \\
\hline
$\chi^4_{24}$ & 
\begin{tabular}{@{}c@{}} 
$F(W_2), F(W_4), F(W_{13}), F(W_{15}),$ \\$F(W_{24}), F(W_{26})$ 
$F(W_{35}), F(W_{37})$
\end{tabular} 
& 
\begin{tabular}{@{}c@{}} 
$e^{\frac{5\pi i}{6}}, e^{-\frac{2\pi i}{3}}, e^{\frac{7\pi i}{12}}, e^{-\frac{11\pi i}{12}}$ \\ 
$e^{\frac{\pi i}{3}}, e^{-\frac{\pi i}{6}}, e^{\frac{7\pi i}{12}}, e^{-\frac{11\pi i}{12}}$
\end{tabular} 
& 8 \\
\hline
$\chi^2_6$ & 
\begin{tabular}{@{}c@{}} 
$(W_{1})_1, (W_{1})_2, (W_{12})_1, (W_{12})_2$ \\ 
$(W_{23})_1, (W_{23})_2, (W_{34})_1, (W_{34})_2$
\end{tabular} 
& 
\begin{tabular}{@{}c@{}} 
$-i, -i, e^{-\frac{3\pi i}{4}}, e^{-\frac{3\pi i}{4}}$ \\ 
$i, i, e^{\frac{\pi i}{3}}, e^{\frac{\pi i}{3}}$
\end{tabular} 
& 8 \\
\hline
\end{tabular}\caption{Twists of simple objects}
\label{TableRibboncondenZ8+8}
\end{table}
Since $\mathcal{Z}(\mathcal{C})_{\mathbb{Z}/2\mathbb Z}^0$ contains a pointed modular subcategory $\C(\mathbb  Z/4\mathbb Z,q)$, there exists a braided tensor equivalence  $\mathcal{Z}(\mathcal{C})_{\mathbb{Z}/2\mathbb Z}^0\cong \C(\mathbb  Z/4\mathbb Z,q)\boxtimes\mathcal{D}$ by \cite[Theorem 3.13]{Drinfeld2010}. Hence, it suffices to determine the $S$-matrix of $\mathcal{D}$. From Table \ref{TableRibboncondenZ8+8},  the set of simple objects of $\mathcal{D}$ is:
\begin{align*}
\mathcal{O}(\mathcal{D})=\{F(\mathbf{1}),F(Y_0),F(Z_{1,7}),(Z_{2,6})_1,(Z_{2,6})_2,F(W_2),F(W_4),(W_1)_1,(W_1)_2\}.
\end{align*}

\begin{lemma}
$S_{(Z_{2,6})_i,F(Z_{1,7})}=-2\chi_6^3$,  $S_{(W_1)_i,F(Z_{1,7})}=0$, $i=1,2$.
\end{lemma}
\begin{proof}
   The $S$-coefficients between  objects $F(Y_0),F(Z_{1,7}), F(W_2),F(W_4)$ are  the same of that in the Drinfeld center by  \cite[Theorem 2.2]{rowell2025neargroup} or \cite[Lemma 4.2]{KirillovOstrik2001}. 
Notice that $\frac{\dim(\mathcal{D})}{\dim (F(Z_{1,7}))^2}$ is an integer and $F(Z_{1,7})$ is the unique simple object with dimension $2\chi_6^3$. Therefore,  $\hat{\sigma}(F(Z_{1,7}))=F(Z_{1,7})$ for all $\sigma\in\operatorname{Gal}(\mathbb{Q}(T)/\mathbb{Q})$, which implies that $\frac{S_{V,F(Z_{1,7})}}{2\chi_6^3}$ is an integer for all $V\in\mathcal{O}(\mathcal{D})$. By \cite[Theorem 2.2]{rowell2025neargroup}, we also have \begin{align*}S_{(Z_{2,6})_1,F(Z_{1,7})}+S_{(Z_{2,6})_2,F(Z_{1,7})}=-4\chi_6^3,~S_{(W_1)_1,F(Z_{1,7})}+S_{(W_1)_2,F(Z_{1,7})}=0.\end{align*} Since $\dim(\mathcal{D})=\sum_{V\in\mathcal{O}(\mathcal{D})} S_{V,F(Z_{1,7})}^2$, it follows that 
\begin{align*}S_{(Z_{2,6})_i,F(Z_{1,7})}=-2\chi_6^3,~ S_{(W_1)_i,F(Z_{1,7})}=0\end{align*}
for $i=1,2$, as claimed.
\end{proof}
\begin{lemma}
$S_{(Z_{2,6})_i,F(W_j)}=0$,  $S_{(W_1)_i,F(W_j)}=\chi_{24}^4$ for $i=1,2$ and $j=2,4$.
\end{lemma}
\begin{proof}
  By  \cite[Theorem 2.2]{rowell2025neargroup}, we have $S_{(Z_{2,6})_1,F(W_j)}+S_{(Z_{2,6})_2,F(W_j)}=0$ and $S_{(W_1)_1,F(W_j)}+S_{(W_1)_2,F(w_j)}=0$. Let $\sigma\in\operatorname{Gal}(\mathbb{Q}(T)/\mathbb{Q})$ such that $\sigma(\sqrt{6})=-\sqrt{6}$, then 
  \begin{align*}
\sigma\left(\frac{S_{F(W_2),F(W_2)}}{\dim(F(W_2))}\right)=-1=\frac{S_{F(W_2),\hat{\sigma}(F(W_2))}}{\dim(\hat{\sigma}(F(W_2)))}.
  \end{align*}
  It follows that  $S_{F(W_2),\hat{\sigma}(F(W_2))}=-\chi_{24}^4$ and $\dim(\hat{\sigma}(F(W_2)))=\chi_{24}^4$. Thus, $\hat{\sigma}(F(W_2))=F(W_2)$ and $\hat{\sigma}(F(W_4))=F(W_4)$. In particular, $\frac{S_{V,F(W_2)}}{\dim(F(W_2))}$ is an integer. Together with the relation $\dim(\mathcal{D})=\sum_{V\in\mathcal{O}(\mathcal{D})} S_{V,F(W_2)}^2$, we have $S_{(Z_{2,6})_i,F(W_2)}=0$,  $S_{(W_1)_i,F(W_2)}=\chi_{24}^4$ for $i=1,2$. The same argument also works for object $F(W_4)$.
\end{proof}
\begin{theorem}\label{ModdatacondeZ8}

  The S-matrix and $T$-matrix of modular category $\mathcal{D}$ are
  \begin{align*}
      S&=\left(\begin{array}{ccccccccc}
     1&\chi_{24}^5&2\chi_6^3&\chi_6^3&\chi_6^3&2\chi_6^2&
     2\chi_6^2 &\chi_6^2&\chi_6^2 \\
     \chi_{24}^5&1&2\chi_6^3&\chi_6^3&\chi_6^3&-2\chi_6^2&
     -2\chi_6^2 &-\chi_6^2&-\chi_6^2\\
     2\chi_6^3& 2\chi_6^3&0&-2\chi_6^3&-2\chi_6^3&0&0&0&0\\
     \chi_6^3&\chi_6^3&-2\chi_6^3
     &u_1&u_2&0&0&u_3&-u_3\\
     \chi_6^3&\chi_6^3&-2\chi_6^3
     &u_2&u_1&0&0&-u_3&u_3\\
    2\chi_6^2&-2\chi_6^2&0&0
     &0&-2\chi_6^2&2\chi_6^2&2\chi_6^2&2\chi_6^2\\
     2\chi_6^2&-\chi_{24}^4&0&0&0&
     2\chi_6^2&
     2\chi_6^2&-2\chi_6^2&
     -2\chi_6^2\\
     \chi_6^2&-\chi_6^2&0&\sqrt{2}\chi_6^3&-\sqrt{2}\chi_6^3
     &2\chi_6^2&-2\chi_6^2&u_5&u_6\\
     \chi_6^2&-\chi_6^2&0&-\sqrt{2}\chi_6^3&\sqrt{2}\chi_6^3
     &2\chi_6^2&-2\chi_6^2&u_6&u_5\\
     \end{array}\right),\\
      T&=\operatorname{diag}(1,1,e^\frac{\pi i}{4},-1,-1,e^\frac{5\pi i}{6},e^\frac{-2\pi i}{3},-e^\frac{2\pi i}{4},-e^\frac{2\pi i}{4}),
  \end{align*}
  where $u_1=\chi_2^1\chi_6^3$, and $u_2=(2-\chi_2^1)\chi_6^3$, $u_5=\chi_3^1\chi_6^2$, $u_6=(2-\chi_3^1)\chi_6^2$, $u_3=\sqrt{2}\chi_6^3$.
\end{theorem}
\begin{proof}
Let \[
u_1 := S_{(Z_{2,6})_1,(Z_{2,6})_1}, \; u_2 := S_{(Z_{2,6})_1,(Z_{2,6})_2}, \; u_3 := S_{(Z_{2,6})_1,(W_1)_1}, \; u_4 := S_{(Z_{2,6})_2,(W_1)_1},
\]
\[
u_5 := S_{(W_1)_1,(W_1)_1}, \; u_6 := S_{(W_1)_1,(W_1)_2}, \; u_7 := S_{(W_1)_2,(W_1)_2}.
\]

Applying orthogonality conditions, we obtain the following:
\begin{itemize} 
    \item  From  $S_{F(Z_{1,7}),-}$ and $ S_{(Z_{2,6})_1,-}$:
$u_1 + u_2 = 2 \chi_6^3$.
    \item From \( S_{(Z_{2,6})_1,-} \) and \( S_{(Z_{2,6})_2,-} \): $ u_1 = S_{(Z_{2,6})_2,(Z_{2,6})_2}.$
    
    \item   By \( S_{(Z_{2,6})_1,-} \) and \( S_{F(W_2),-} \): $ S_{(Z_{2,6})_1,(W_1)_2} = -u_3.$
    \item  Similarly, by $S_{(Z_{2,6})_2,-}$ and $ S_{F(W_2),-}$ we have: $S_{(Z_{2,6})_2,(W_1)_2} = -u_4.$

    \item From $S_{F(Z_{1,7}),-}$ and $ S_{(W_1)_1,-}$:
    $u_3+u_4=0$.

    \item From $S_{(W_1)_1,-}$ and $ S_{F(W_2),-}$, we get $ u_5+u_6=2\chi_6^2$.
    \item By $S_{(W_1)_2,-}$ and $ S_{F(W_2),-}$:   $u_7+u_6=2 \chi_6^2, \text { implying } u_5=u_7.$\item Finally, orthogonality between $S_{(W_1)_1,-}$ and $S_{(Z_{2,6})_1,-}$ gives $u_3(u_1-u_2+u_5-u_6)=0.$
\end{itemize}

Note that $u_3$ can't be zero. Otherwise, $u_1u_2=-3(\chi_6^3)^2$ and $u_5u_6=-10(5+2\sqrt{6})$, then  $N_{F(Y_0),(W_1)_2}^{(W_1)_2}=-\frac{1}{2}$, which gives a contradiction. Therefore, $(u_1-u_2+u_5-u_6)=0$.   We claim that $(W_1)_1,(W_1)_2$ must be in the same orbit under the action of $\operatorname{Gal}(\mathcal C)$. 
Suppose not.  Then for any  simple object $X$ and all $\sigma\in\operatorname{Gal}(\mathbb{Q}(T)/\mathbb{Q})$, $\sigma\left(\frac{S_{X,(W_1)_1}}{\chi_6^2}\right)=\frac{S_{X,(W_1)_1}}{\chi_6^2}$, implying  $S_{X,(W_1)_1}=a_X\chi_6^2$ for some integer $a_X$. Let $S_{(Z_{2,6})_1, (W_1)_1}=a\chi_6^2$ and $S_{(W_1)_1,(W_1)_1}=b\chi_6^2$. This leads to the equation $2a^2+b^2+(2-b)^2=14$, which has no integer solution.  Notice that 
\begin{align*}
\sigma\bigg(\frac{S_{(Z_{2,6)_1},(W_1)_1}}{\chi_6^2}\bigg)=\frac{S_{(Z_{2,6)_1},(W_1)_2}}{\chi_6^2}=-\frac{S_{(Z_{2,6})_1, (W_1)_1}}{\chi_6^2}
\end{align*}
and $\sigma\left(\frac{u_5}{\chi_6^2}\right)=
\frac{u_6}{\chi_6^2}$.
By using a similar argument, we conclude that  $\frac{u_3}{\chi_6^3}$, $\frac{u_1}{\chi_6^3}$ and $\frac{u_2}{\chi_6^3}$ cannot be integers. Thus simple objects $(Z_{2,6})_1$ and $(Z_{2,6})_2$  are in the same Galois orbit, and for some $\tau\in\operatorname{Gal}(\mathbb{Q}(T)/\mathbb{Q})$,  $\tau\left(\frac{u_1}{\chi_6^3}\right)=\frac{u_2}{\chi_6^3}$. 

Applying the Verlinde formula, we obtain
\begin{align*}
    N_{(W_1)_1,(W_1)_1}^{F(W_4)}=\frac{1}{\dim(\mathcal{D})}\sum_{V\in\mathcal{O}(\mathcal{D})}
\frac{S_{(W_1)_1,V}^2S_{F(W_4),V}}{\dim(V)}=\frac{2}{3}-\frac{u_5^2+u_6^2}{12(\chi_6^2)^2}.
\end{align*}

If $u_5$ and $u_6$ are real algebraic integers, notice that $u_5,u_6 \neq 0$, and the fusion coefficient must be non-negative. It follows that $u_5^2+u_6^2=8(\chi_6^2)^2$. Combine this  with $u_5+u_6=\chi_{24}^4$, we have $u_5,u_6=\chi_6^2\pm\sqrt{3}\chi_6^2$.   Therefore, we find  that $u_3=\pm\sqrt{2}\chi_6^3$, as desired.  

If $u_5^2+u_6^2$ is negative, i.e., $\overline{u_5}=u_6$, then $((W_1)_1)^*=(W_1)_2$, and $u_5^2+u_6^2=-4(3k+1)(\chi_6^2)^2$ for some integer $k\geq0$. Additionally, $\dim(\mathcal{D})=10(\chi_6^2)^2-2u_3^2+2|u_5|^2$, which implies $|u_5|^2-u_3^2=7(\chi_6^2)^2$. By the balancing equation, we have 
\begin{align*}
    |u_5|\leq \dim((W_1)_1)^2,|u_6|\leq \dim((W_1)_1)\dim((W_1)_2),\end{align*} 
    which then implies the following inequality
\begin{align*}
   2(\chi_6^2
   )^4\geq 2|u_5|^2\geq |u_5^2+u_6^2|=4(3k+1)(\chi_6^2)^2.
\end{align*}
Hence, we have $k\leq 2$. We now show this case cannot occur. In fact, observe that $u_5,u_6$ are roots of  equation $x^2-2\chi_6^2x+\frac{1}{2}((2\chi_6^2)^2-(u_5^2+u_6^2))=0$. Then 
\begin{align*}
  N_{(W_1)_1,(W_1)_1}^{(W_1)_1}=\frac{1}{\dim(\mathcal{D})}\sum_{V\in\mathcal{O}(\mathcal{D})}
\frac{S_{(W_1)_1,V}^2S_{(W_1)_2,V}}{\dim(V)}=\frac{k+1}{2},   
\end{align*}
which implies  $k=1$. If $k=1$, then
\begin{align*}
 N_{(W_1)_1,(W_1)_1}^{F(W_2)}=\frac{1}{\dim(\mathcal{D})}\sum_{V\in\mathcal{O}(\mathcal{D})}
\frac{S_{(W_1)_1,V}^2S_{F(W_2),V}}{\dim(V)}=-1,   
\end{align*}
 which is a contradiction. Hence $u_5,u_6,u_3$ must be real algebraic integers. Similarly, by  determining the fusion coefficient $N_{(Z_{2,6)})_1,(Z_{2,6)_1}}^{F(Z_{1,7})}$, we obtain
\begin{align*}
  u_1^2+u_2^2=8(\chi_6^3)^2, ~u_1u_2=-(\chi_6^3)^2 
\end{align*}
which yields $u_1,u_2=\chi_6^3\pm \sqrt{2}\chi_6^3$. Thus,  we recover the desired modular data.
\end{proof}   

The following modular data of $\mathcal{C}(\mathfrak{g}_2,4)$ can be produced  using Sagemath \cite{sage-fusion-rings}
\begin{align*}
S=&\left(
\begin{array}{ccccccccc}
 1 & \chi_6^2 & \chi_6^3 & 2 \chi_6^2 & 2 \chi_6^3 & \chi_6^2 & \chi_{24}^5 & 2\chi_6^2 & \chi_6^3 \\
 \chi_6^2 & \chi_3^1\chi_6^2 & \sqrt{2}\chi_6^3 & 2 \chi_6^2 & 0 & (2-\chi_3^1)\chi_6^2 & -\chi_6^2 & -2 \chi_6^2 & -\sqrt{2}\chi_6^3  \\
 \chi_6^3 & \sqrt{2}\chi_6^3 & (2-\chi_2^1)\chi_6^3 & 0 & -2 \chi_6^3& -\sqrt{2}\chi_6^3 & \chi_6^3 & 0 &\chi_2^1\chi_6^3 \\
 2 \chi_6^2 & 2 \chi_6^2 & 0 & -2 \chi_6^2 & 0 & 2 \chi_6^2 & -2 \chi_6^2 & 2\chi_6^2 & 0 \\
 2\chi_6^3 & 0 & -2\chi_6^3& 0 & 0 & 0 & 2 \chi_6^3 & 0 & -2\chi_6^3\\
 \chi_6^2 & (2-\chi_3^1)\chi_6^2 & -\sqrt{2}\chi_6^3 & 2 \chi_6^2 & 0 & \chi_3^1\chi_6^2 & -\chi_6^2 & -2  \chi_6^2 & \sqrt{2}\chi_6^3 \\
  \chi_{24}^5& -\chi_6^2 & \chi_6^3 & -2  \chi_6^2 & 2  \chi_6^3  & -\chi_6^2 & 1 & -2  \chi_6^2 & \chi_6^3 \\
 2 \chi_6^2  & -2  \chi_6^2  & 0 & 2 \chi_6^2 & 0 & -2 \chi_6^2  & -2 \chi_6^2 & 2\chi_6^2 & 0 \\
 \chi_6^3 & -\sqrt{2}\chi_6^3 & \chi_2^1\chi_6^3 & 0 & -2 \chi_6^3  & \sqrt{2}\chi_6^3 & \chi_6^3 & 0 & (2-\chi_2^1)\chi_6^3 \\
\end{array}
\right)\\
    T=&\operatorname{diag}\left(1,\zeta_4,-1,\zeta_{12}^7,-\zeta_8^3,\zeta_4,1,\zeta_3,-1
\right).
\end{align*}
\begin{corollary}\label{realcondZ8}
    Modular tensor category $\mathcal{D}$ has the same modular data as  $\C(\mathfrak{g}_2,4)$.
\end{corollary}
\begin{proof}
  If we replace $\zeta_8$ and $\zeta_3$ by their  inverses in modular category $\C(\g_2,4)$ (i.e., replace $\zeta_{24}$ by its inverse), then we get the expected ribbon twists. 
\end{proof}
\section*{Acknowledgments}
The authors thank Zhenghan Wang for pointing out that the modular data  in $\left(\ref{eq:MDrank10}\right)$ can potentially be realized as the center of a fusion category. They also thank Yilong Wang for helpful discussions. The second author further thanks Eric Rowell, Siu-Hung Ng, Pinhas Grossman, and Andrew Schopieray for their valuable discussions during the SLMath Quantum Symmetry Reunion in 2024. The first author is supported by the National Natural Science Foundation of China (No. 12101541).
 \vspace{3.5mm}
 
\appendix

\section{Solutions for $(\omega_j, \tau_j, \xi_j)$}\label{Appendix}
In this section, we list the solutions of $\left(\omega_j, \tau_j, \xi_j\right)$ from Equations $\left(\ref{eq:half1}\right)$ -- $\left(\ref{eq:half4}\right)$ for the near-group categories in Section \ref{sec:Z4Z4} and Section \ref{sec:Z8}.

The following table gives the solutions  $\left(\omega_j, \tau_j, \xi_j\right)$ for the near-group category $\mathbb{Z}/4\mathbb{Z}\times \mathbb{Z}/4\mathbb{Z}+ 16$, where $j=1, \cdots, 152$.  Each $\omega_j$ is of the form $\zeta_{80}^k$ for some integer $k$, with the corresponding $k$ values provided in the relevant column. Since $\xi_j(x)=\exp\left(i\theta_{j,x}\right)$, the table records the values $\theta_{j,x}$   in the column for $\xi$, ordered by $x= (0,0),\cdots, (0,3),(1,0),\cdots,(3,3)\in \mathbb Z/4\mathbb Z\times \mathbb Z/4\mathbb Z$.
\begin{center}\small{
   \begin{longtable}{|c|c|c|c|}
   \hline
    $\#$ & $\omega$ & $\tau$  & $\xi$ \\
    \hline
    \endfirsthead

    \hline
    $\#$ & $\omega$ & $\tau$  & $\xi$ \\
    \hline
    \endhead
    
    \hline
    \multicolumn{4}{r}{\textit{Continued on next page}} \\
    \endfoot
    \endlastfoot
 
        1 & $32$ & $(0,0)$ & \makecell[c]{$1.25664,0.13318,-1.88496,-2.33229,-2.55613,-0.706082,1.11454,0.35208, $ \\$-1.88496,-1.57477,  -1.88496,-0.62434,-2.78457,2.16119,-0.172061,-3.06383$ }\\
        \hline
        2 & $32$ & $(0,0)$ & \makecell[c]{$-1.88496,-2.33229,1.25664,0.13318,1.11454, 0.35208,-2.55613,-0.706082,$ \\$-1.88496,-0.62434, -1.88496,-1.57477,-0.172061,-3.06383,-2.78457,2.16119$} \\
        \hline
        3 & $32$ & $(0,0)$ & \makecell[c]{$-1.88496,-0.62434,-1.88496,-1.57477,-0.172061 -3.06383,-2.78457,2.16119,$ \\$-1.88496,-2.33229, 1.25664,0.13318,1.11454,0.35208,-2.55613,-0.706082$ }\\
        \hline
        4 & $32$ & $(0,0)$ & \makecell[c]{$-1.88496,-1.57477,-1.88496,-0.62434,-2.78457, 2.16119,-0.172061,-3.06383,$ \\$1.25664,0.13318, -1.88496,-2.33229,-2.55613,-0.706082,1.11454,0.35208$ }\\
        \hline   5 & $48$ & $(0,0)$ & \makecell[c]{$ 1.88496,-1.11454,1.88496,0.172061,2.33229, -0.35208,0.62434,3.06383,$ \\$-1.25664,2.55613, 1.88496,2.78457,-0.13318,0.706082,1.57477,-2.16119$} \\
        \hline
        6 & $48$ & $(0,0)$ & \makecell[c]{$1.88496,-1.11454,1.88496,0.172061,2.33229,-0.35208,0.62434,3.06383, $ \\$-1.25664,2.55613, 1.88496,2.78457,-0.13318,0.706082,1.57477,-2.16119 $} \\
        \hline
        7 & $48$ & $(0,0)$ & \makecell[c]{$1.88496,0.172061,1.88496,-1.11454,0.62434, 3.06383,2.33229,-0.35208,$ \\$ 1.88496,2.78457, -1.25664,2.55613,1.57477,-2.16119,-0.13318,0.706082 $ }\\
        \hline
        8 & $48$ & $(0,0)$ & \makecell[c]{$1.88496,2.78457,-1.25664,2.55613,1.57477, -2.16119,-0.13318,0.706082, $ \\$1.88496,0.172061, 1.88496,-1.11454,0.62434,3.06383,2.33229,-0.35208 $ }\\
        \hline 
         9 & $7 $ & $(0,1)$ & \makecell[c]{$ -0.883402,-0.923013,3.0903,-1.75512,-3.07932, -2.22175,-1.62208,1.30524, $ \\$0.974834,-2.78125, -0.625598,1.96077,-1.15546,-0.297893,-1.54086,1.38646  $} \\
          \hline
          10& $7 $ & $(0,1)$ & \makecell[c]{$0.974834,-2.78125,-0.625598,1.96077,-1.15546, -0.297893,-1.54086,1.38646,   $ \\$-0.883402,-0.923013, 3.0903,-1.75512,-3.07932,-2.22175,-1.62208,1.30524 $ }\\
          \hline
           11 & $15 $ & $(0,1)$ &  \makecell[c]{$ -2.88492,1.70682,-0.169554,2.13305,-0.937357, -1.81154,0.572197,-0.179498,  $ \\$-2.88492,1.70682, -0.169554,2.13305,-0.937357,-1.81154,0.572197,-0.179498 $} \\
          \hline
           12 & $23 $ & $(0,1)$ & \makecell[c]{$0.521724,-1.0715,-0.173026,2.76484,2.2799, 2.94948,-1.23134,-1.42571, $ \\$ -1.38096,0.83118, 1.16544,1.42637,1.21313,1.88271,2.44672,2.25236$} \\
          \hline
           13 & $23 $ & $(0,1)$ & \makecell[c]{$ -1.38096,0.83118,1.16544,1.42637,1.21313, 1.88271,2.44672,2.25236,  $ \\$0.521724,-1.0715, -0.173026,2.76484,2.2799,2.94948,-1.23134,-1.42571 $} \\
          \hline
           14 & $47 $ & $(0,1)$ & \makecell[c]{$-0.625598,1.96077,0.974834,-2.78125,-1.54086, 1.38646,-1.15546,-0.297893,  $ \\$ 3.0903,-1.75512, -0.883402,-0.923013,-1.62208,1.30524,-3.07932,-2.22175 $} \\
          \hline
          15 & $47  $ & $(0,1)$ & \makecell[c]{$ 3.0903,-1.75512,-0.883402,-0.923013,-1.62208, 1.30524,-3.07932,-2.22175, $ \\$-0.625598,1.96077, 0.974834,-2.78125,-1.54086,1.38646,-1.15546,-0.297893  $ }\\
          \hline
          16 & $55 $ & $(0,1)$ & \makecell[c]{$-0.169554,2.13305,-2.88492,1.70682,0.572197,  -0.179498,-0.937357,-1.81154,  $ \\$ -0.169554,2.13305, -2.88492,1.70682,0.572197,-0.179498,-0.937357,-1.81154$} \\
          \hline
          17 & $63 $ & $(0,1)$ & \makecell[c]{$ 1.16544,1.42637,-1.38096,0.83118,2.44672, 2.25236,1.21313,1.88271,  $ \\$-0.173026,2.76484, 0.521724,-1.0715,-1.23134,-1.42571,2.2799,2.94948 $ }\\
          \hline
          18 & $ 63$ & $(0,1)$ & \makecell[c]{$ -0.173026,2.76484,0.521724,-1.0715,-1.23134, -1.42571,2.2799,2.94948, $ \\$1.16544,1.42637, -1.38096,0.83118,2.44672,2.25236,1.21313,1.88271  $} \\
          \hline
          19 & $12 $ & $(0,2)$ & \makecell[c]{$ 0.13318,-1.88496,-2.33229,1.25664,-0.706082,1.11454,0.35208,-2.55613, $ \\$ -1.57477,-1.88496,-0.62434,-1.88496,2.16119,-0.172061,-3.06383,-2.78457   $} \\
          \hline
            20 & $12 $ & $(0,2)$ & \makecell[c]{$ -1.57477,-1.88496,-0.62434,-1.88496,2.16119,-0.172061,-3.06383,-2.78457, $ \\$    0.13318,-1.88496,-2.33229,1.25664,-0.706082,1.11454,0.35208,-2.55613$} \\
          \hline
            21 & $12 $ & $(0,2)$ & \makecell[c]{$ -0.62434,-1.88496,-1.57477,-1.88496,-3.06383,-2.78457,2.16119,-0.172061, $ \\$   -2.33229,1.25664,0.13318,-1.88496,0.35208,-2.55613,-0.706082,1.11454$} \\
          \hline
           22 & $12 $ & $(0,2)$ & \makecell[c]{$ -2.33229,1.25664,0.13318,-1.88496,0.35208,-2.55613,-0.706082,1.11454, $ \\$ -0.62434,-1.88496,-1.57477,-1.88496,-3.06383,-2.78457,2.16119,-0.172061  $} \\
          \hline
          23 & $28 $ & $(0,2)$ & \makecell[c]{$  2.55613,1.88496,2.78457,-1.25664,0.706082,1.57477,-2.16119,-0.13318,$ \\$-1.11454,1.88496,0.172061,1.88496,-0.35208,0.62434,3.06383,2.33229  $} \\
          \hline
          24 & $28 $ & $(0,2)$ & \makecell[c]{$ 0.172061,1.88496,-1.11454,1.88496,3.06383,2.33229,-0.35208,0.62434, $ \\$ 2.78457,-1.25664,2.55613,1.88496,-2.16119,-0.13318,0.706082,1.57477  $} \\
          \hline
            25 & $28 $ & $(0,2)$ & \makecell[c]{$ 2.78457,-1.25664,2.55613,1.88496,-2.16119,-0.13318,0.706082,1.57477, $ \\$  0.172061,1.88496,-1.11454,1.88496,3.06383,2.33229,-0.35208,0.62434  $} \\
          \hline
            26 & $ 28$ & $(0,2)$ & \makecell[c]{$ -1.11454,1.88496,0.172061,1.88496,-0.35208,0.62434,3.06383,2.33229, $ \\$ 2.55613,1.88496,2.78457,-1.25664,0.706082,1.57477,-2.16119,-0.13318  $} \\
          \hline
          27 & $7 $ & $(0,3)$ & \makecell[c]{$ -2.78125,-0.625598,1.96077,0.974834,-0.297893,-1.54086,1.38646,-1.15546, $ \\$-0.923013,3.0903,-1.75512,-0.883402,-2.22175,-1.62208,1.30524,-3.07932   $} \\
          \hline
          28 & $ 7$ & $(0,3)$ & \makecell[c]{$ -0.923013,3.0903,-1.75512,-0.883402,-2.22175,-1.62208,1.30524,-3.07932, $ \\$  -2.78125,-0.625598,1.96077,0.974834,-0.297893,-1.54086,1.38646,-1.15546  $} \\
          \hline
          29 & $15$ & $(0,3)$ & \makecell[c]{$ 1.70682,-0.169554,2.13305,-2.88492,-1.81154,0.572197,-0.179498,-0.937357, $ \\$  1.70682,-0.169554,2.13305,-2.88492,-1.81154,0.572197,-0.179498,-0.937357  $} \\
          \hline
          30 & $23 $ & $(0,3)$ & \makecell[c]{$0.83118,1.16544,1.42637,-1.38096,1.88271,2.44672,2.25236,1.21313,  $ \\$  -1.0715,-0.173026,2.76484,0.521724,2.94948,-1.23134,-1.42571,2.2799  $} \\
          \hline
          31 & $ 23$ & $(0,3)$ & \makecell[c]{$ -1.0715,-0.173026,2.76484,0.521724,2.94948,-1.23134,-1.42571,2.2799, $ \\$  0.83118,1.16544,1.42637,-1.38096,1.88271,2.44672,2.25236,1.21313 $} \\
          \hline
          32 & $47 $ & $(0,3)$ & \makecell[c]{$ 1.96077,0.974834,-2.78125,-0.625598,1.38646,-1.15546,-0.297893,-1.54086, $ \\$-1.75512,-0.883402,-0.923013,3.0903,1.30524,-3.07932,-2.22175,-1.62208   $} \\
          \hline
          33 & $47 $ & $(0,3)$ & \makecell[c]{$ -1.75512,-0.883402,-0.923013,3.0903,1.30524,-3.07932,-2.22175,-1.62208, $ \\$ 1.96077,0.974834,-2.78125,-0.625598,1.38646,-1.15546,-0.297893,-1.54086  $} \\
          \hline
          34 & $ 55$ & $(0,3)$ & \makecell[c]{$2.13305,-2.88492,1.70682,-0.169554,-0.179498,-0.937357,-1.81154,0.572197,  $ \\$    2.13305,-2.88492,1.70682,-0.169554,-0.179498,-0.937357,-1.81154,0.572197$} \\
          \hline
          35 & $ 63$ & $(0,3)$ & \makecell[c]{$2.76484,0.521724,-1.0715,-0.173026,-1.42571,2.2799,2.94948,-1.23134,  $ \\$  1.42637,-1.38096,0.83118,1.16544,2.25236,1.21313,1.88271,2.44672  $} \\
          \hline
          36 & $63 $ & $(0,3)$ & \makecell[c]{$ 1.42637,-1.38096,0.83118,1.16544,2.25236,1.21313,1.88271,2.44672, $ \\$ 2.76484,0.521724,-1.0715,-0.173026,-1.42571,2.2799,2.94948,-1.23134   $} \\
          \hline
           37 & $17 $ & $(1,0)$ & \makecell[c]{$ -1.16544,-2.44672,0.173026,1.23134,-1.42637,-2.25236,-2.76484,1.42571, $ \\$  1.38096,-1.21313,-0.521724,-2.2799,-0.83118,-1.88271,1.0715,-2.94948  $} \\
          \hline
           38 & $17 $ & $(1,0)$ & \makecell[c]{$ 0.173026,1.23134,-1.16544,-2.44672,-2.76484,1.42571,-1.42637,-2.25236, $ \\$ -0.521724,-2.2799,1.38096,-1.21313,1.0715,-2.94948,-0.83118,-1.88271   $} \\
          \hline
           39 & $25 $ & $(1,0)$ & \makecell[c]{$ 0.169554,-0.572197,0.169554,-0.572197,-2.13305,0.179498,-2.13305,0.179498, $ \\$   2.88492,0.937357,2.88492,0.937357,-1.70682,1.81154,-1.70682,1.81154$} \\
          \hline
           40 & $33 $ & $(1,0)$ & \makecell[c]{$ 0.625598,1.54086,-3.0903,1.62208,-1.96077,-1.38646,1.75512,-1.30524, $ \\$ -0.974834,1.15546,0.883402,3.07932,2.78125,0.297893,0.923013,2.22175  $} \\
          \hline
           41 & $33 $ & $(1,0)$ & \makecell[c]{$-3.0903,1.62208,0.625598,1.54086,1.75512,-1.30524,-1.96077,-1.38646,  $ \\$ 0.883402,3.07932,-0.974834,1.15546,0.923013,2.22175,2.78125,0.297893   $} \\
          \hline
           42 & $ 57$ & $(1,0)$ & \makecell[c]{$ -0.521724,-2.2799,1.38096,-1.21313,1.0715,-2.94948,-0.83118,-1.88271, $ \\$  0.173026,1.23134,-1.16544,-2.44672,-2.76484,1.42571,-1.42637,-2.25236  $} \\
          \hline
           43 & $ 57$ & $(1,0)$ & \makecell[c]{$1.38096,-1.21313,-0.521724,-2.2799,-0.83118,-1.88271,1.0715,-2.94948,  $ \\$ -1.16544,-2.44672,0.173026,1.23134,-1.42637,-2.25236,-2.76484,1.42571  $} \\
          \hline
           44 & $ 65$ & $(1,0)$ & \makecell[c]{$2.88492,0.937357,2.88492,0.937357,-1.70682,1.81154,-1.70682,1.81154,  $ \\$  0.169554,-0.572197,0.169554,-0.572197,-2.13305,0.179498,-2.13305,0.179498 $} \\
          \hline
           45 & $73 $ & $(1,0)$ & \makecell[c]{$0.883402,3.07932,-0.974834,1.15546,0.923013,2.22175,2.78125,0.297893,  $ \\$-3.0903,1.62208,0.625598,1.54086,1.75512,-1.30524,-1.96077,-1.38646   $} \\
          \hline
           46 & $ 73$ & $(1,0)$ & \makecell[c]{$-0.974834,1.15546,0.883402,3.07932,2.78125,0.297893,0.923013,2.22175,  $ \\$    0.625598,1.54086,-3.0903,1.62208,-1.96077,-1.38646,1.75512,-1.30524$} \\
          \hline
           47 & $ 12$ & $(1,1)$ & \makecell[c]{$2.14417,0.33666,0.50165,2.34429,0.605818,-1.20169,1.73978,-2.70077,  $ \\$3.05334,1.2107,-2.63533,2.03946,2.87337,1.03073,-1.09698,-2.70538   $} \\
          \hline
           48 & $12 $ & $(1,1)$ & \makecell[c]{$-2.63533,2.03946,3.05334,1.2107,-1.09698,-2.70538,2.87337,1.03073,  $ \\$ 0.50165,2.34429,2.14417,0.33666,1.73978,-2.70077,0.605818,-1.20169   $} \\
          \hline
           49 & $20$ & $(1,1)$ & \makecell[c]{$-1.23768,-1.23768,1.23768,-2.80847,2.80847,2.80847,1.23768,-2.80847,  $ \\$ 1.23768,-2.80847,-1.23768,-1.23768,1.23768,-2.80847,2.80847,2.80847 $} \\
          \hline
           50 & $ 28$ & $(1,1)$ & \makecell[c]{$-3.05334,-2.87337,-2.14417,-0.605818,-1.2107,-1.03073,-0.33666,1.20169,$ \\$ 2.63533,1.09698,-0.50165,-1.73978,-2.03946,2.70538,-2.34429,2.70077 $} \\
          \hline
           51 & $ 28$ & $(1,1)$ & \makecell[c]{$-0.50165,-1.73978,2.63533,1.09698,-2.34429,2.70077,-2.03946,2.70538,  $ \\$ -2.14417,-0.605818,-3.05334,-2.87337,-0.33666,1.20169,-1.2107,-1.03073 $} \\
          \hline
          
           52 & $52 $ & $(1,1)$ & \makecell[c]{$ 3.05334,1.2107,-2.63533,2.03946,2.87337,1.03073,-1.09698,-2.70538, $ \\$ 2.14417,0.33666,0.50165,2.34429,0.605818,-1.20169,1.73978,-2.70077 $} \\
          \hline
           53 & $52 $ & $(1,1)$ & \makecell[c]{$ 0.50165,2.34429,2.14417,0.33666,1.73978,-2.70077,0.605818,-1.20169, $ \\$ -2.63533,2.03946,3.05334,1.2107,-1.09698,-2.70538,2.87337,1.03073  $} \\
          \hline
           54 & $60 $ & $(1,1)$ & \makecell[c]{$ 1.23768,-2.80847,-1.23768,-1.23768,1.23768,-2.80847,2.80847,2.80847, $ \\$  -1.23768,-1.23768,1.23768,-2.80847,2.80847,2.80847,1.23768,-2.80847  $} \\
          \hline
           55 & $68 $ & $(1,1)$ & \makecell[c]{$ -2.14417,-0.605818,-3.05334,-2.87337,-0.33666,1.20169,-1.2107,-1.03073, $ \\$ -0.50165,-1.73978,2.63533,1.09698,-2.34429,2.70077,-2.03946,2.70538  $} \\
          \hline
           56 & $ 68$ & $(1,1)$ & \makecell[c]{$ 2.63533,1.09698,-0.50165,-1.73978,-2.03946,2.70538,-2.34429,2.70077, $ \\$ -3.05334,-2.87337,-2.14417,-0.605818,-1.2107,-1.03073,-0.33666,1.20169  $} \\
          \hline
           57 & $5 $ & $(1,2)$ & \makecell[c]{$ -0.572197,0.169554,-0.572197,0.169554,0.179498,-2.13305,0.179498,-2.13305, $ \\$  0.937357,2.88492,0.937357,2.88492,1.81154,-1.70682,1.81154,-1.70682  $} \\
          \hline
           58 & $13 $ & $(1,2)$ & \makecell[c]{$ 1.54086,-3.0903,1.62208,0.625598,-1.38646,1.75512,-1.30524,-1.96077, $ \\$  1.15546,0.883402,3.07932,-0.974834,0.297893,0.923013,2.22175,2.78125  $} \\
          \hline
           59 & $13 $ & $(1,2)$ & \makecell[c]{$1.62208,0.625598,1.54086,-3.0903,-1.30524,-1.96077,-1.38646,1.75512,  $ \\$  3.07932,-0.974834,1.15546,0.883402,2.22175,2.78125,0.297893,0.923013 $} \\
          \hline
           60 & $ 37$ & $(1,2)$ & \makecell[c]{$-2.2799,1.38096,-1.21313,-0.521724,-2.94948,-0.83118,-1.88271,1.0715,  $ \\$  1.23134,-1.16544,-2.44672,0.173026,1.42571,-1.42637,-2.25236,-2.76484 $} \\
          \hline
           61 & $ 37$ & $(1,2)$ & \makecell[c]{$-1.21313,-0.521724,-2.2799,1.38096,-1.88271,1.0715,-2.94948,-0.83118,  $ \\$    -2.44672,0.173026,1.23134,-1.16544,-2.25236,-2.76484,1.42571,-1.42637$} \\
          \hline
           62 & $45 $ & $(1,2)$ & \makecell[c]{$ 0.937357,2.88492,0.937357,2.88492,1.81154,-1.70682,1.81154,-1.70682, $ \\$-0.572197,0.169554,-0.572197,0.169554,0.179498,-2.13305,0.179498,-2.13305  $} \\
          \hline
           63 & $43 $ & $(1,2)$ & \makecell[c]{$1.15546,0.883402,3.07932,-0.974834,0.297893,0.923013,2.22175,2.78125,  $ \\$ 1.54086,-3.0903,1.62208,0.625598,-1.38646,1.75512,-1.30524,-1.96077  $} \\
          \hline
           64 & $53 $ & $(1,2)$ & \makecell[c]{$ 3.07932,-0.974834,1.15546,0.883402,2.22175,2.78125,0.297893,0.923013, $ \\$ 1.62208,0.625598,1.54086,-3.0903,-1.30524,-1.96077,-1.38646,1.75512  $} \\
          \hline
           65 & $77 $ & $(1,2)$ & \makecell[c]{$ 1.23134,-1.16544,-2.44672,0.173026,1.42571,-1.42637,-2.25236,-2.76484, $ \\$  -2.2799,1.38096,-1.21313,-0.521724,-2.94948,-0.83118,-1.88271,1.0715 $} \\
          \hline
           66 & $77$ & $(1,2)$ & \makecell[c]{$-2.44672,0.173026,1.23134,-1.16544,-2.25236,-2.76484,1.42571,-1.42637, $ \\$  -1.21313,-0.521724,-2.2799,1.38096,-1.88271,1.0715,-2.94948,-0.83118 $} \\
          \hline
           67 & $12 $ & $(1,3)$ & \makecell[c]{$0.33666,0.50165,2.34429,2.14417,-1.20169,1.73978,-2.70077,0.605818,$ \\$  1.2107,-2.63533,2.03946,3.05334,1.03073,-1.09698,-2.70538,2.87337  $} \\
          \hline
           68 & $12 $ & $(1,3)$ & \makecell[c]{$2.03946,3.05334,1.2107,-2.63533,-2.70538,2.87337,1.03073,-1.09698, $ \\$  2.34429,2.14417,0.33666,0.50165,-2.70077,0.605818,-1.20169,1.73978 $} \\
          \hline
           69 & $20 $ & $(1,3)$ & \makecell[c]{$-1.23768,1.23768,-2.80847,-1.23768,2.80847,1.23768,-2.80847,2.80847, $ \\$ -2.80847,-1.23768,-1.23768,1.23768,-2.80847,2.80847,2.80847,1.23768 $} \\
          \hline
           70 & $ 28$ & $(1,3)$ & \makecell[c]{$ -2.87337,-2.14417,-0.605818,-3.05334,-1.03073,-0.33666,1.20169,-1.2107, $ \\$  1.09698,-0.50165,-1.73978,2.63533,2.70538,-2.34429,2.70077,-2.03946  $} \\
          \hline
           71 & $ 28$ & $(1,3)$ & \makecell[c]{$ -1.73978,2.63533,1.09698,-0.50165,2.70077,-2.03946,2.70538,-2.34429, $ \\$ -0.605818,-3.05334,-2.87337,-2.14417,1.20169,-1.2107,-1.03073,-0.33666$} \\
          \hline
           72 & $ 52$ & $(1,3)$ & \makecell[c]{$ 2.34429,2.14417,0.33666,0.50165,-2.70077,0.605818,-1.20169,1.73978, $ \\$ 2.03946,3.05334,1.2107,-2.63533,-2.70538,2.87337,1.03073,-1.09698 $} \\
          \hline
           73 & $ 52 $ & $(1,3)$ & \makecell[c]{$1.2107,-2.63533,2.03946,3.05334,1.03073,-1.09698,-2.70538,2.87337,  $ \\$ 0.33666,0.50165,2.34429,2.14417,-1.20169,1.73978,-2.70077,0.605818  $} \\
          \hline
           74 & $ 60$ & $(1,3)$ & \makecell[c]{$ -2.80847,-1.23768,-1.23768,1.23768,-2.80847,2.80847,2.80847,1.23768, $ \\$ -1.23768,1.23768,-2.80847,-1.23768,2.80847,1.23768,-2.80847,2.80847  $} \\
          \hline
           75 & $68 $ & $(1,3)$ & \makecell[c]{$-0.605818,-3.05334,-2.87337,-2.14417,1.20169,-1.2107,-1.03073,-0.33666,  $ \\$-1.73978,2.63533,1.09698,-0.50165,2.70077,-2.03946,2.70538,-2.34429   $} \\
          \hline
           76 & $68  $ & $(1,3)$ & \makecell[c]{$1.09698,-0.50165,-1.73978,2.63533,2.70538,-2.34429,2.70077,-2.03946,  $ \\$ -2.87337,-2.14417,-0.605818,-3.05334,-1.03073,-0.33666,1.20169,-1.2107  $} \\
          \hline
           77 & $ 52$ & $(2,0)$ & \makecell[c]{$ -2.55613,-0.706082,1.11454,0.35208,-1.88496,-1.57477,-1.88496,-0.62434, $ \\$  -2.78457,2.16119,-0.172061,-3.06383,1.25664,0.13318,-1.88496,-2.33229 $} \\
          \hline
           78 & $ 52$ & $(2,0)$ & \makecell[c]{$-0.172061,-3.06383,-2.78457,2.16119,-1.88496,-2.33229,1.25664,0.13318,  $ \\$ 1.11454,0.35208,-2.55613,-0.706082,-1.88496,-0.62434,-1.88496,-1.57477   $} \\
          \hline
           79 & $52 $ & $(2,0)$ & \makecell[c]{$ -2.78457,2.16119,-0.172061,-3.06383,1.25664,0.13318,-1.88496,-2.33229, $ \\$ -2.55613,-0.706082,1.11454,0.35208,-1.88496,-1.57477,-1.88496,-0.62434  $} \\
          \hline
           80 & $52 $ & $(2,0)$ & \makecell[c]{$ 1.11454,0.35208,-2.55613,-0.706082,-1.88496,-0.62434,-1.88496,-1.57477, $ \\$  -0.172061,-3.06383,-2.78457,2.16119,-1.88496,-2.33229,1.25664,0.13318 $} \\
          \hline
           81 & $ 68$ & $(2,0)$ & \makecell[c]{$ -0.13318,0.706082,1.57477,-2.16119,1.88496,-1.11454,1.88496,0.172061, $ \\$2.33229,-0.35208,0.62434,3.06383,-1.25664,2.55613,1.88496,2.78457  $} \\
          \hline
           82 & $68 $ & $(2,0)$ & \makecell[c]{$ 1.57477,-2.16119,-0.13318,0.706082,1.88496,0.172061,1.88496,-1.11454, $ \\$  0.62434,3.06383,2.33229,-0.35208,1.88496,2.78457,-1.25664,2.55613 $} \\
          \hline
           83 & $68 $ & $(2,0)$ & \makecell[c]{$ 0.62434,3.06383,2.33229,-0.35208,1.88496,2.78457,-1.25664,2.55613, $ \\$ 1.57477,-2.16119,-0.13318,0.706082,1.88496,0.172061,1.88496,-1.11454  $} \\
          \hline
           84 & $ 68$ & $(2,0)$ & \makecell[c]{$ 2.33229,-0.35208,0.62434,3.06383,-1.25664,2.55613,1.88496,2.78457, $ \\$ -0.13318,0.706082,1.57477,-2.16119,1.88496,-1.11454,1.88496,0.172061  $} \\
          \hline
           85 & $3 $ & $(2,1)$ & \makecell[c]{$-1.23134,-1.42571,2.2799,2.94948,1.16544,1.42637,-1.38096,0.83118,  $ \\$ 2.44672,2.25236,1.21313,1.88271,-0.173026,2.76484,0.521724,-1.0715 $} \\
          \hline
           86 & $3 $ & $(2,1)$ & \makecell[c]{$ 2.44672,2.25236,1.21313,1.88271,-0.173026,2.76484,0.521724,-1.0715, $ \\$ -1.23134,-1.42571,2.2799,2.94948,1.16544,1.42637,-1.38096,0.83118  $} \\
          \hline
           87 & $ 27$ & $(2,1)$ & \makecell[c]{$-1.15546,-0.297893,-1.54086,1.38646,-0.883402,-0.923013,3.0903,-1.75512,  $ \\$ -3.07932,-2.22175,-1.62208,1.30524,0.974834,-2.78125,-0.625598,1.96077  $} \\
          \hline
           88 & $27 $ & $(2,1)$ & \makecell[c]{$ -3.07932,-2.22175,-1.62208,1.30524,0.974834,-2.78125,-0.625598,1.96077, $ \\$   -1.15546,-0.297893,-1.54086,1.38646,-0.883402,-0.923013,3.0903,-1.75512 $} \\
          \hline
           89 & $35 $ & $(2,1)$ & \makecell[c]{$  -0.937357,-1.81154,0.572197,-0.179498,-2.88492,1.70682,-0.169554,2.13305,$ \\$ -0.937357,-1.81154,0.572197,-0.179498,-2.88492,1.70682,-0.169554,2.13305  $} \\
          \hline
           90 & $ 43$ & $(2,1)$ & \makecell[c]{$2.2799,2.94948,-1.23134,-1.42571,-1.38096,0.83118,1.16544,1.42637,  $ \\$  1.21313,1.88271,2.44672,2.25236,0.521724,-1.0715,-0.173026,2.76484  $} \\
          \hline
           91 & $43 $ & $(2,1)$ & \makecell[c]{$ 1.21313,1.88271,2.44672,2.25236,0.521724,-1.0715,-0.173026,2.76484, $ \\$  2.2799,2.94948,-1.23134,-1.42571,-1.38096,0.83118,1.16544,1.42637  $} \\
          \hline
           92 & $67 $ & $(2,1)$ & \makecell[c]{$ -1.54086,1.38646,-1.15546,-0.297893,3.0903,-1.75512,-0.883402,-0.923013, $ \\$-1.62208,1.30524,-3.07932,-2.22175,-0.625598,1.96077,0.974834,-2.78125  $} \\
          \hline
           93 & $67 $ & $(2,1)$ & \makecell[c]{$ -1.62208,1.30524,-3.07932,-2.22175,-0.625598,1.96077,0.974834,-2.78125, $ \\$  -1.54086,1.38646,-1.15546,-0.297893,3.0903,-1.75512,-0.883402,-0.923013  $} \\
          \hline
           94 & $ 75$ & $(2,1)$ & \makecell[c]{$ 0.572197,-0.179498,-0.937357,-1.81154,-0.169554,2.13305,-2.88492,1.70682, $ \\$ 0.572197,-0.179498,-0.937357,-1.81154,-0.169554,2.13305,-2.88492,1.70682  $} \\
          \hline
           95 & $ 32$ & $(2,2)$ & \makecell[c]{$ -0.706082,1.11454,0.35208,-2.55613,-1.57477,-1.88496,-0.62434,-1.88496, $ \\$  2.16119,-0.172061,-3.06383,-2.78457,0.13318,-1.88496,-2.33229,1.25664  $} \\
          \hline
           96 & $32 $ & $(2,2)$ & \makecell[c]{$ -3.06383,-2.78457,2.16119,-0.172061,-2.33229,1.25664,0.13318,-1.88496, $ \\$ 0.35208,-2.55613,-0.706082,1.11454,-0.62434,-1.88496,-1.57477,-1.88496  $} \\
          \hline
           97 & $32 $ & $(2,2)$ & \makecell[c]{$ 2.16119,-0.172061,-3.06383,-2.78457,0.13318,-1.88496,-2.33229,1.25664, $ \\$  -0.706082,1.11454,0.35208,-2.55613,-1.57477,-1.88496,-0.62434,-1.88496  $} \\
          \hline
           98 & $ 32$ & $(2,2)$ & \makecell[c]{$ 0.35208,-2.55613,-0.706082,1.11454,-0.62434,-1.88496,-1.57477,-1.88496, $ \\$  -3.06383,-2.78457,2.16119,-0.172061,-2.33229,1.25664,0.13318,-1.88496  $} \\
          \hline
           99 & $48 $ & $(2,2)$ & \makecell[c]{$ 0.706082,1.57477,-2.16119,-0.13318,-1.11454,1.88496,0.172061,1.88496, $ \\$ -0.35208,0.62434,3.06383,2.33229,2.55613,1.88496,2.78457,-1.25664  $} \\
          \hline
           100 & $ 48$ & $(2,2)$ & \makecell[c]{$ -2.16119,-0.13318,0.706082,1.57477,0.172061,1.88496,-1.11454,1.88496, $ \\$3.06383,2.33229,-0.35208,0.62434,2.78457,-1.25664,2.55613,1.88496  $} \\
          \hline
           101 & $48 $ & $(2,2)$ & \makecell[c]{$3.06383,2.33229,-0.35208,0.62434,2.78457,-1.25664,2.55613,1.88496,  $ \\$  -2.16119,-0.13318,0.706082,1.57477,0.172061,1.88496,-1.11454,1.88496  $} \\
          \hline
           102 & $ 48$ & $(2,2)$ & \makecell[c]{$ -0.35208,0.62434,3.06383,2.33229,2.55613,1.88496,2.78457,-1.25664, $ \\$  0.706082,1.57477,-2.16119,-0.13318,-1.11454,1.88496,0.172061,1.88496 $} \\
          \hline
           103 & $3 $ & $(2,3)$ & \makecell[c]{$ 2.25236,1.21313,1.88271,2.44672,2.76484,0.521724,-1.0715,-0.173026, $ \\$ -1.42571,2.2799,2.94948,-1.23134,1.42637,-1.38096,0.83118,1.16544$} \\
          \hline
           104 & $ 3$ & $(2,3)$ & \makecell[c]{$-1.42571,2.2799,2.94948,-1.23134,1.42637,-1.38096,0.83118,1.16544,  $ \\$  2.25236,1.21313,1.88271,2.44672,2.76484,0.521724,-1.0715,-0.173026 $} \\
          \hline
           105 & $ 27$ & $(2,3)$ & \makecell[c]{$-0.297893,-1.54086,1.38646,-1.15546,-0.923013,3.0903,-1.75512,-0.883402,  $ \\$ -2.22175,-1.62208,1.30524,-3.07932,-2.78125,-0.625598,1.96077,0.974834$} \\
          \hline
           106 & $ 27$ & $(2,3)$ & \makecell[c]{$  -2.22175,-1.62208,1.30524,-3.07932,-2.78125,-0.625598,1.96077,0.974834,$ \\$ -0.297893,-1.54086,1.38646,-1.15546,-0.923013,3.0903,-1.75512,-0.883402 $} \\
          \hline
           107 & $ 35$ & $(2,3)$ & \makecell[c]{$ -1.81154,0.572197,-0.179498,-0.937357,1.70682,-0.169554,2.13305,-2.88492, $ \\$ -1.81154,0.572197,-0.179498,-0.937357,1.70682,-0.169554,2.13305,-2.88492$} \\
          \hline
           108 & $43 $ & $(2,3)$ & \makecell[c]{$1.88271,2.44672,2.25236,1.21313,-1.0715,-0.173026,2.76484,0.521724,  $ \\$ 2.94948,-1.23134,-1.42571,2.2799,0.83118,1.16544,1.42637,-1.38096  $} \\
          \hline
           109 & $43 $ & $(2,3)$ & \makecell[c]{$2.94948,-1.23134,-1.42571,2.2799,0.83118,1.16544,1.42637,-1.38096,  $ \\$ 1.88271,2.44672,2.25236,1.21313,-1.0715,-0.173026,2.76484,0.521724  $} \\
          \hline
           110 & $67 $ & $(2,3)$ & \makecell[c]{$ 1.30524,-3.07932,-2.22175,-1.62208,1.96077,0.974834,-2.78125,-0.625598, $ \\$  1.38646,-1.15546,-0.297893,-1.54086,-1.75512,-0.883402,-0.923013,3.0903  $} \\
          \hline
           111 & $67 $ & $(2,3)$ & \makecell[c]{$1.38646,-1.15546,-0.297893,-1.54086,-1.75512,-0.883402,-0.923013,3.0903,  $ \\$  1.30524,-3.07932,-2.22175,-1.62208,1.96077,0.974834,-2.78125,-0.625598$} \\
          \hline
           112 & $ 75$ & $(2,3)$ & \makecell[c]{$-0.179498,-0.937357,-1.81154,0.572197,2.13305,-2.88492,1.70682,-0.169554,  $ \\$ -0.179498,-0.937357,-1.81154,0.572197,2.13305,-2.88492,1.70682,-0.169554   $} \\
          \hline
           113 & $17 $ & $(3,0)$ & \makecell[c]{$ -2.76484,1.42571,-1.42637,-2.25236,-0.521724,-2.2799,1.38096,-1.21313, $ \\$ 1.0715,-2.94948,-0.83118,-1.88271,0.173026,1.23134,-1.16544,-2.44672 $} \\
          \hline
           114 & $ 17$ & $(3,0)$ & \makecell[c]{$ -1.42637,-2.25236,-2.76484,1.42571,1.38096,-1.21313,-0.521724,-2.2799, $ \\$ -0.83118,-1.88271,1.0715,-2.94948,-1.16544,-2.44672,0.173026,1.23134   $} \\
          \hline
           115 & $25 $ & $(3,0)$ & \makecell[c]{$-2.13305,0.179498,-2.13305,0.179498,2.88492,0.937357,2.88492,0.937357,  $ \\$-1.70682,1.81154,-1.70682,1.81154,0.169554,-0.572197,0.169554,-0.572197  $} \\
          \hline
           116 & $ 33$ & $(3,0)$ & \makecell[c]{$ -1.96077,-1.38646,1.75512,-1.30524,-0.974834,1.15546,0.883402,3.07932, $ \\$ 2.78125,0.297893,0.923013,2.22175,0.625598,1.54086,-3.0903,1.62208  $} \\
          \hline
           117 & $33 $ & $(3,0)$ & \makecell[c]{$  1.75512,-1.30524,-1.96077,-1.38646,0.883402,3.07932,-0.974834,1.15546,$ \\$  0.923013,2.22175,2.78125,0.297893,-3.0903,1.62208,0.625598,1.54086 $} \\
          \hline
           118 & $ 57$ & $(3,0)$ & \makecell[c]{$ -0.83118,-1.88271,1.0715,-2.94948,-1.16544,-2.44672,0.173026,1.23134, $ \\$-1.42637,-2.25236,-2.76484,1.42571,1.38096,-1.21313,-0.521724,-2.2799  $} \\
          \hline
           119 & $ 57$ & $(3,0)$ & \makecell[c]{$  1.0715,-2.94948,-0.83118,-1.88271,0.173026,1.23134,-1.16544,-2.44672,$ \\$ -2.76484,1.42571,-1.42637,-2.25236,-0.521724,-2.2799,1.38096,-1.21313   $} \\
          \hline
           120 & $ 65$ & $(3,0)$ & \makecell[c]{$ -1.70682,1.81154,-1.70682,1.81154,0.169554,-0.572197,0.169554,-0.572197, $ \\$-2.13305,0.179498,-2.13305,0.179498,2.88492,0.937357,2.88492,0.937357   $} \\
          \hline
           121 & $ 73$ & $(3,0)$ & \makecell[c]{$ 2.78125,0.297893,0.923013,2.22175,0.625598,1.54086,-3.0903,1.62208, $ \\$  -1.96077,-1.38646,1.75512,-1.30524,-0.974834,1.15546,0.883402,3.07932  $} \\
          \hline
           122 & $ 73$ & $(3,0)$ & \makecell[c]{$ 0.923013,2.22175,2.78125,0.297893,-3.0903,1.62208,0.625598,1.54086, $ \\$ 1.75512,-1.30524,-1.96077,-1.38646,0.883402,3.07932,-0.974834,1.15546 $} \\
          \hline
           123 & $12 $ & $(3,1)$ & \makecell[c]{$  0.605818,-1.20169,1.73978,-2.70077,3.05334,1.2107,-2.63533,2.03946,$ \\$ 2.87337,1.03073,-1.09698,-2.70538,2.14417,0.33666,0.50165,2.34429 $} \\
          \hline
           124 & $12 $ & $(3,1)$ & \makecell[c]{$ -1.09698,-2.70538,2.87337,1.03073,0.50165,2.34429,2.14417,0.33666, $ \\$ 1.73978,-2.70077,0.605818,-1.20169,-2.63533,2.03946,3.05334,1.2107 $} \\
          \hline
           125 & $ 20$ & $(3,1)$ & \makecell[c]{$  2.80847,2.80847,1.23768,-2.80847,1.23768,-2.80847,-1.23768,-1.23768,$ \\$ 1.23768,-2.80847,2.80847,2.80847,-1.23768,-1.23768,1.23768,-2.80847  $} \\
          \hline
           126 & $28 $ & $(3,1)$ & \makecell[c]{$ -2.34429,2.70077,-2.03946,2.70538,-2.14417,-0.605818,-3.05334,-2.87337, $ \\$-0.33666,1.20169,-1.2107,-1.03073,-0.50165,-1.73978,2.63533,1.09698  $} \\
          \hline
           127 & $28 $ & $(3,1)$ & \makecell[c]{$  -1.2107,-1.03073,-0.33666,1.20169,2.63533,1.09698,-0.50165,-1.73978,$ \\$  -2.03946,2.70538,-2.34429,2.70077,-3.05334,-2.87337,-2.14417,-0.605818  $} \\
          \hline
           128 & $ 52$ & $(3,1)$ & \makecell[c]{$  2.87337,1.03073,-1.09698,-2.70538,2.14417,0.33666,0.50165,2.34429,$ \\$0.605818,-1.20169,1.73978,-2.70077,3.05334,1.2107,-2.63533,2.03946 $} \\
          \hline
           129 & $52 $ & $(3,1)$ & \makecell[c]{$  1.73978,-2.70077,0.605818,-1.20169,-2.63533,2.03946,3.05334,1.2107,$ \\$ -1.09698,-2.70538,2.87337,1.03073,0.50165,2.34429,2.14417,0.33666  $} \\
          \hline
           130 & $60 $ & $(3,1)$ & \makecell[c]{$ 1.23768,-2.80847,2.80847,2.80847,-1.23768,-1.23768,1.23768,-2.80847, $ \\$ 2.80847,2.80847,1.23768,-2.80847,1.23768,-2.80847,-1.23768,-1.23768  $} \\
          \hline
           131 & $ 68$ & $(3,1)$ & \makecell[c]{$-0.33666,1.20169,-1.2107,-1.03073,-0.50165,-1.73978,2.63533,1.09698,  $ \\$ -2.34429,2.70077,-2.03946,2.70538,-2.14417,-0.605818,-3.05334,-2.87337   $} \\
          \hline
           132 & $ 68$ & $(3,1)$ & \makecell[c]{$ -2.03946,2.70538,-2.34429,2.70077,-3.05334,-2.87337,-2.14417,-0.605818, $ \\$ -1.2107,-1.03073,-0.33666,1.20169,2.63533,1.09698,-0.50165,-1.73978  $} \\
          \hline
           133 & $ 5$ & $(3,2)$ & \makecell[c]{$ 0.179498,-2.13305,0.179498,-2.13305,0.937357,2.88492,0.937357,2.88492, $ \\$  1.81154,-1.70682,1.81154,-1.70682,-0.572197,0.169554,-0.572197,0.169554 $} \\
          \hline
           134 & $ 13$ & $(3,2)$ & \makecell[c]{$-1.30524,-1.96077,-1.38646,1.75512,3.07932,-0.974834,1.15546,0.883402,  $ \\$ 2.22175,2.78125,0.297893,0.923013,1.62208,0.625598,1.54086,-3.0903 $} \\
          \hline
           135 & $ 13$ & $(3,2)$ & \makecell[c]{$  -1.38646,1.75512,-1.30524,-1.96077,1.15546,0.883402,3.07932,-0.974834,$ \\$0.297893,0.923013,2.22175,2.78125,1.54086,-3.0903,1.62208,0.625598 $} \\
          \hline
           136 & $ 37$ & $(3,2)$ & \makecell[c]{$ -1.88271,1.0715,-2.94948,-0.83118,-2.44672,0.173026,1.23134,-1.16544, $ \\$-2.25236,-2.76484,1.42571,-1.42637,-1.21313,-0.521724,-2.2799,1.38096   $} \\
          \hline
           137 & $ 37$ & $(3,2)$ & \makecell[c]{$  -2.94948,-0.83118,-1.88271,1.0715,1.23134,-1.16544,-2.44672,0.173026,$ \\$  1.42571,-1.42637,-2.25236,-2.76484,-2.2799,1.38096,-1.21313,-0.521724  $} \\
          \hline
           138 & $ 45$ & $(3,2)$ & \makecell[c]{$1.81154,-1.70682,1.81154,-1.70682,-0.572197,0.169554,-0.572197,0.169554,  $ \\$ 0.179498,-2.13305,0.179498,-2.13305,0.937357,2.88492,0.937357,2.88492  $} \\
          \hline
           139 & $53 $ & $(3,2)$ & \makecell[c]{$0.297893,0.923013,2.22175,2.78125,1.54086,-3.0903,1.62208,0.625598,  $ \\$  -1.38646,1.75512,-1.30524,-1.96077,1.15546,0.883402,3.07932,-0.974834  $} \\
          \hline
           140 & $53 $ & $(3,2)$ & \makecell[c]{$2.22175,2.78125,0.297893,0.923013,1.62208,0.625598,1.54086,-3.0903,  $ \\$ -1.30524,-1.96077,-1.38646,1.75512,3.07932,-0.974834,1.15546,0.883402  $} \\
          \hline
           141 & $ 77$ & $(3,2)$ & \makecell[c]{$  -2.25236,-2.76484,1.42571,-1.42637,-1.21313,-0.521724,-2.2799,1.38096,$ \\$ -1.88271,1.0715,-2.94948,-0.83118,-2.44672,0.173026,1.23134,-1.16544  $} \\
          \hline
           142 & $77 $ & $(3,2)$ & \makecell[c]{$ 1.42571,-1.42637,-2.25236,-2.76484,-2.2799,1.38096,-1.21313,-0.521724, $ \\$ -2.94948,-0.83118,-1.88271,1.0715,1.23134,-1.16544,-2.44672,0.173026   $} \\
          \hline
           143 & $12 $ & $(3,3)$ & \makecell[c]{$-2.70538,2.87337,1.03073,-1.09698,2.34429,2.14417,0.33666,0.50165,  $ \\$   -2.70077,0.605818,-1.20169,1.73978,2.03946,3.05334,1.2107,-2.63533 $} \\
          \hline
           144 & $12 $ & $(3,3)$ & \makecell[c]{$ -1.20169,1.73978,-2.70077,0.605818,1.2107,-2.63533,2.03946,3.05334, $ \\$  1.03073,-1.09698,-2.70538,2.87337,0.33666,0.50165,2.34429,2.14417  $} \\
          \hline
           145 & $20 $ & $(3,3)$ & \makecell[c]{$  2.80847,1.23768,-2.80847,2.80847,-2.80847,-1.23768,-1.23768,1.23768,$ \\$  -2.80847,2.80847,2.80847,1.23768,-1.23768,1.23768,-2.80847,-1.23768 $} \\
          \hline
           146 & $28 $ & $(3,3)$ & \makecell[c]{$-1.03073,-0.33666,1.20169,-1.2107,1.09698,-0.50165,-1.73978,2.63533,  $ \\$ 2.70538,-2.34429,2.70077,-2.03946,-2.87337,-2.14417,-0.605818,-3.05334  $} \\
          \hline
           147 & $ 28$ & $(3,3)$ & \makecell[c]{$ 2.70077,-2.03946,2.70538,-2.34429,-0.605818,-3.05334,-2.87337,-2.14417, $ \\$  1.20169,-1.2107,-1.03073,-0.33666,-1.73978,2.63533,1.09698,-0.50165  $} \\
          \hline
           148 & $52 $ & $(3,3)$ & \makecell[c]{$1.03073,-1.09698,-2.70538,2.87337,0.33666,0.50165,2.34429,2.14417,  $ \\$ -1.20169,1.73978,-2.70077,0.605818,1.2107,-2.63533,2.03946,3.05334  $} \\
          \hline
           149 & $ 52$ & $(3,3)$ & \makecell[c]{$ -2.70077,0.605818,-1.20169,1.73978,2.03946,3.05334,1.2107,-2.63533, $ \\$ -2.70538,2.87337,1.03073,-1.09698,2.34429,2.14417,0.33666,0.50165  $} \\
          \hline
           150 & $ 60$ & $(3,3)$ & \makecell[c]{$ -2.80847,2.80847,2.80847,1.23768,-1.23768,1.23768,-2.80847,-1.23768, $ \\$ 2.80847,1.23768,-2.80847,2.80847,-2.80847,-1.23768,-1.23768,1.23768  $} \\
          \hline
           151 & $ 68$ & $(3,3)$ & \makecell[c]{$2.70538,-2.34429,2.70077,-2.03946,-2.87337,-2.14417,-0.605818,-3.05334,  $ \\$-1.03073,-0.33666,1.20169,-1.2107,1.09698,-0.50165,-1.73978,2.63533    $} \\
          \hline
           152 & $ 68$ & $(3,3)$ & \makecell[c]{$ 1.20169,-1.2107,-1.03073,-0.33666,-1.73978,2.63533,1.09698,-0.50165, $ \\$ 2.70077,-2.03946,2.70538,-2.34429,-0.605818,-3.05334,-2.87337,-2.14417   $} \\
          \hline

             \caption{$\left(\omega_j, \tau_j, \xi_j\right)$ for $\mathbb{Z}/4\mathbb{Z}\times \mathbb{Z}/4\mathbb{Z}+ 16$}
    \label{table:solZ4Z4}
 \end{longtable} }
\end{center}

The table below gives the solutions $\left(\omega_j, \tau_j, \xi_j\right)$ in Table \ref{DataZ8} for the near-group category $J_8^1$ of type $\mathbb Z/8\mathbb Z+8$, where $j=1, \cdots, 44$. The $\omega_i$ equals to $\zeta_{48}^k$ for some $k \in \mathbb{Z}$, which are recorded in the corresponding column. 
 \begin{center}\small{
   \begin{longtable}{|c|c|c|c|}
   \hline
    $\#$ & $\omega$ & $\tau$  & $\xi$ \\
    \hline
    \endfirsthead

    \hline
    $\#$ & $\omega$ & $\tau$  & $\xi$ \\
    \hline
    \endhead
    
    \hline
    \multicolumn{4}{r}{\textit{Continued on next page}} \\
    \endfoot
    \endlastfoot
 
  1  & $24$     & $0$ & $1.8326, 0.840155, -2.87979, -2.67275, 1.8326, 0.840155, -2.87979,  -2.67275$       \\\hline
2  &$40$   & $0$ & $-2.35619, -2.03783, -2.19092, 2.32649, -2.35619, 0.0297022, 0.620122, -1.88916 $           \\\hline
3  & $40$    & $0$ & $-2.35619, 0.0297022, 0.620122, -1.88916, -2.35619, -2.03783, -2.19092, 2.32649  $        \\\hline
4  & $16 $    & $0$ &$ -1.5708, 1.76553, -1.06745, 0.914113, 1.5708, 3.01288, 1.06745, 2.16146   $        \\ \hline
5  &$16 $  & $0$ & $1.5708, 3.01288, 1.06745, 2.16146, -1.5708, 1.76553, -1.06745, \
0.914113$  \\\hline
6  &$ 39 $   & $1$ & $2.86523, -0.967181, -2.38552, -2.68973, 0.113755, -1.3573, 3.01698, \
-0.428819$ \\\hline
7  & $39$    & $1$ & $0.113755, -1.3573, 3.01698, -0.428819, 2.86523, -0.967181, -2.38552, \
-2.68973 $ \\\hline
8  & $7$     & $1$ & $-0.960866, 0.764517, 0.941698, 2.72147, -2.6727, -0.665243, 1.79457, 0.432749 $ \\\hline
9  & $7$    & $1$ & $0.7022, -0.89855, -0.94287, -2.38689, 0.955559, 1.98968, 0.619742, \
2.31732 $\\\hline
10 & $7  $   & $1$ &$ -2.6727, -0.665243, 1.79457, 0.432749, -0.960866, 0.764517, 0.941698, \
2.72147$ \\ \hline
11 & $7$     & $1$ & $0.955559, 1.98968, 0.619742, 2.31732, 0.7022, -0.89855, -0.94287, \
-2.38689 $   \\\hline
12 & $12$     & $2$ & $-2.67275, 1.8326, 0.840155, -2.87979, -2.67275, 1.8326, 0.840155, \
-2.87979$          \\\hline
13 & $28 $   & $2$ & $2.32649, -2.35619, 0.0297022, 0.620122, -1.88916, -2.35619, -2.03783, \
-2.19092  $       \\\hline
14 & $28$    &$2$ & $-1.88916, -2.35619, -2.03783, -2.19092, 2.32649, -2.35619, 0.0297022, \
0.620122   $      \\ \hline
15 & $4 $  & $2$ &$ 2.16146, -1.5708, 1.76553, -1.06745, 0.914113, 1.5708, 3.01288, \
1.06745 $ \\\hline
16 & $4 $  & $2$ & $0.914113, 1.5708, 3.01288, 1.06745, 2.16146, -1.5708, 1.76553, \
-1.06745 $\\\hline
17 & $15$     & $3$ & $-0.428819, 2.86523, -0.967181, -2.38552, -2.68973, 0.113755, -1.3573, \
3.01698 $\\\hline
18 &$ 15 $    & $3$ &$ -2.68973, 0.113755, -1.3573, 3.01698, -0.428819, 2.86523, -0.967181, \
-2.38552 $\\\hline
19 & $31$     & $3$ & $2.72147, -2.6727, -0.665243, 1.79457, 0.432749, -0.960866, 0.764517, \
0.941698 $ \\\hline
20 & $31$     & $3$ &$ -2.38689, 0.955559, 1.98968, 0.619742, 2.31732, 0.7022, -0.89855, \
-0.94287 $\\ \hline
21 & $31  $   & $3$ & $2.31732, 0.7022, -0.89855, -0.94287, -2.38689, 0.955559, 1.98968, \
0.619742$ \\\hline
22 & $31$     & $3$ & $0.432749, -0.960866, 0.764517, 0.941698, 2.72147, -2.6727, -0.665243, \
1.79457$\\\hline
23 & $24 $   & $4$ &$ -2.87979, -2.67275, 1.8326, 0.840155, -2.87979, -2.67275, 1.8326, \
0.840155 $\\\hline
24 & $16$    & $4$ & $1.06745, 2.16146, -1.5708, 1.76553, -1.06745, 0.914113, 1.5708, \
3.01288$ \\\hline
25 & $16$    &$4$ & $-1.06745, 0.914113, 1.5708, 3.01288, 1.06745, 2.16146, -1.5708, \
1.76553  $ \\\hline
26 & $40 $    & $4 $& $-2.19092, 2.32649, -2.35619, 0.0297022, 0.620122, -1.88916, -2.35619, \
-2.03783 $  \\ \hline
27 &$ 40$    & $4$ &$ 0.620122, -1.88916, -2.35619, -2.03783, -2.19092, 2.32649, -2.35619, \
0.0297022 $ \\\hline
28 & $15 $    & $5$ &$ -2.38552, -2.68973, 0.113755, -1.3573, 3.01698, -0.428819, 2.86523, \
-0.967181$ \\\hline
29 & $15 $    & $5$& $3.01698, -0.428819, 2.86523, -0.967181, -2.38552, -2.68973, 0.113755, \
-1.3573$\\\hline
30 & $31  $  & $5$ & $1.79457, 0.432749, -0.960866, 0.764517, 0.941698, 2.72147, -2.6727, \
-0.665243$\\\hline
31 & $31$    & $5$ &$ 0.619742, 2.31732, 0.7022, -0.89855, -0.94287, -2.38689, 0.955559, \
1.98968$\\\hline
32 & $31 $   & $5$ & $-0.94287, -2.38689, 0.955559, 1.98968, 0.619742, 2.31732, 0.7022, \
-0.89855 $\\ \hline
33 & $31$    & $5$& $0.941698, 2.72147, -2.6727, -0.665243, 1.79457, 0.432749, -0.960866, \
0.764517$\\\hline
34 & $12$    &$6$& $0.840155, -2.87979, -2.67275, 1.8326, 0.840155, -2.87979, -2.67275, \
1.8326 $\\\hline
35 & $28$   & $6$ & $-2.03783, -2.19092, 2.32649, -2.35619, 0.0297022, 0.620122, -1.88916, \
-2.35619 $\\\hline
36 & $28$    & $6 $& $0.0297022, 0.620122, -1.88916, -2.35619, -2.03783, -2.19092, 2.32649, \
-2.35619 $\\\hline
37 & $4$   &$ 6 $& $3.01288, 1.06745, 2.16146, -1.5708, 1.76553, -1.06745, 0.914113, \
1.5708  $  \\\hline
38 & $4$     & $6$&$ 1.76553, -1.06745, 0.914113, 1.5708, 3.01288, 1.06745, 2.16146, \
-1.5708 $  \\ \hline
39 & $39$     & $7$&$ -0.967181, -2.38552, -2.68973, 0.113755, -1.3573, 3.01698, -0.428819, \
2.86523$ \\\hline
40 & $39$  &$ 7$ &$ -1.3573, 3.01698, -0.428819, 2.86523, -0.967181, -2.38552, -2.68973, \
0.113755$ \\\hline
41 & $7$     &$ 7$ &$ -0.89855, -0.94287, -2.38689, 0.955559, 1.98968, 0.619742, 2.31732, \
0.7022$\\\hline
42 &$7$      & $7$ & $0.764517, 0.941698, 2.72147, -2.6727, -0.665243, 1.79457, 0.432749, \
-0.960866$\\\hline
43 &$7$    & $7$ & $1.98968, 0.619742, 2.31732, 0.7022, -0.89855, -0.94287, -2.38689, \
0.955559 $\\\hline
44 &   $7$  & $7$  &  $-0.665243, 1.79457, 0.432749, -0.960866, 0.764517, 0.941698, 2.72147, \
-2.6727 $\\ \hline

\caption{$(\omega_j, \tau_j, \xi_j)$ for $J_{8}^1$ in Table \ref{DataZ8}.}
\label{table:triplesforJ81}
 \end{longtable} }
\end{center}
 


\bibliographystyle{abbrv}
\bibliography{zbib}

\end{document}